\documentclass[12pt]{amsart}

\usepackage{amsmath, amssymb, amsthm}
\usepackage{mathtools}
\usepackage[utf8]{inputenc}
\usepackage[T1]{fontenc}
\usepackage{lmodern}
\usepackage{microtype}

\usepackage{graphicx}
\usepackage{subcaption}
\usepackage{booktabs}
\usepackage{multirow}

\usepackage[sort&compress,numbers]{natbib}
\usepackage{hyperref}
\hypersetup{
	colorlinks=true,
	linkcolor=blue,
	citecolor=blue,
	urlcolor=blue,
	pdfauthor={M. T. K. Abassi, A. Acharqy, K. Boutahir, Y. Rami},
	pdftitle={An algorithm to determine LS-category and Ginsburg invariant of any rationally elliptic space}
}

\usepackage{algorithm}
\usepackage{algpseudocode}

\theoremstyle{plain}
\newtheorem{theorem}{Theorem}[section]

\newtheorem{proposition}[theorem]{Proposition}

\theoremstyle{definition}

\newtheorem{example}{Example}[section]

\theoremstyle{remark}
\newtheorem{remark}{Remark}[section]

\usepackage[margin=1in]{geometry}
\usepackage[english]{babel}

\title{An algorithm to determine LS-category and Ginsburg invariant of any rationally elliptic space}

\author[M. T. K. Abassi]{M. T. K. Abassi$^1$}
\address{$^1$Department of Mathematics, Sidi Mohamed Ben Abdellah University, Fès, Morocco}
\email{mohamed.abassi@usmba.ac.ma}

\author[A. Acharqy]{A. Acharqy$^2$}
\address{$^2$Department of Mathematics, Sidi Mohamed Ben Abdellah University, Fès, Morocco}
\email{abdelouhid.acharqy@usmba.ac.ma}

\author[K. Boutahir]{K. Boutahir$^3$}
\address{$^3$Department of Mathematics, Moulay Ismaïl University, Meknès, Morocco}
\email{k.boutahir@umi.ac.ma}

\author[Y. Rami]{Y. Rami$^4$}
\address{$^4$Department of Mathematics, Moulay Ismaïl University, Meknès, Morocco}
\email{y.rami@umi.ac.ma}


\begin{document}
	
	\begin{abstract}
	Let $X$ be a rationally elliptic space. Utilizing the Gorenstein algebra structure of $X$, we present 
	 three algorithms that together induce a generating class of $Ext^N_{(\Lambda V,d)}(\mathbb{Q},(\Lambda V,d))$ with $N$ being the formal dimension of $X$.
	 From these algorithms, we derive an algorithm to compute the rational Lusternik-Schnirelmann category $cat_0(X)$.
	  Furthermore, by applying a spectral sequence argument based on the {\it Eilenberg-Moore spectral sequence}, we compute the rational Ginsburg invariant $l_0(X)$ introduced by M. Ginsburg in \cite{Gin}.
	\end{abstract}

	\maketitle

	\section{Introduction}
	\label{intro}

Rational Homotopy Theory (RHT) studies homotopy theory "modulo torsion groups", through localization with respect to the empty set of primes. This construction associates to every simply connected space $X$ its \emph{rationalization}, denoted $X_{0}$, with the property that $\pi_n(X)\otimes \mathbb{Q}  \cong \pi_n(X_{0})$ for any integer $n$. Two such  spaces $X$ and $Y$ are said to be rationally homotopy equivalent, written $X \sim_{\mathbb{Q}} Y$, if their rationalizations $X_{0}$ and $Y_{0}$ are homotopy equivalent. Thus, RHT is the study of spaces up to rational homotopy equivalence, and $X$ is said to be \emph{rational} if the $\mathbb{Z}$-module $\pi_n(X)$ (or equivalently $H_{n}(X, \mathbb{Z})$)  is a $\mathbb{Q}$-vector space for each $n\geq 1$. Ignoring torsion in homotopy and homology groups of $X$ sacrifices many properties (e.g., those arising from Steenrod operations), but it yields a complete algebraic invariant of its rational equivalence class.

Before presenting our main results, we recall that  one of the most spectacular mathematical applications of RHT states that a sufficiently general simply connected Riemannian manifold has infinitely many closed geodesics \cite{VPS76}. We also cite the Allday-Halperin Theorem : “if a torus T acts freely in a homogeneous space $G/H$, then $\dim T \leq rank (G)-rank (H)$” \cite{Al-HA}. 
A strong appearance of RHT in Mathematical Physics, H. Sati's Hypothesis H suggests that the Sullivan minimal model of the $4$-sphere $\mathbb{S}^4$ captures the dynamics of fields in $M$-theory in the low-energy (UV) limit, also known as $11d$ supergravity \cite{Sat18}.

Within the theory, a fundamental result by Y. Félix and S. Halperin states that finite 1-connected CW-complexes and more generally spaces of finite Lusternik-Schnirelmann category cat($X$), are naturally distributed into two distinct classes: the {\it elliptic ones}  and the {\it hyperbolic ones}. The former are characterized by the fact that the ranks of their homotopy groups are almost all zero. Moreover, they all satisfy the Poincaré duality property over $\mathbb{Q}$ and their  Euler characteristics are non-negative. Recall that $X$ is said to be \emph{Poincaré duality space} over $\mathbb{Q}$ if its graded rational cohomology algebra $H^*(X,\mathbb{Q})$ satisfies the Poincaré duality property (cf. \S 2 below). For instance,  in Riemannian geometry, the simply connected Dupin hypersurfaces in $\mathbb{S}^n$ are elliptic. Moreover all known examples of simply connected positively curved manifolds are elliptic, and a conjecture of Bott asserts that they all should be \cite{FHT91, FHT93}.

Our work leverages Sullivan's algebraic approach to RHT, specifically the \emph{Sullivan minimal model} defined in \S 2. We will focus on spaces $X$ having the homotopy type of simply connected finite type CW-complexes. Each such space has a Sullivan minimal model i.e. a free commutative differential graded algebra  $(\Lambda V,d)$, where $V$ is a graded finite type vector space, which is equipped with a decomposable differential $d = d_k + d_{k+1} + \dots$ ($k\geq 2$).   This model satisfies  $H^*(X,\mathbb{Q})\cong H^*(\Lambda V,d)$ and $V \cong Hom_{\mathbb{Z}}(\pi_*(X),\mathbb{Q})$ \cite{Sul78}. In this context, $X$ is  (rationally) {\it elliptic} if and only if $V$ and  $H^*(\Lambda V,d)$ are both finite dimensional. Each elliptic space satisfies the Poincaré duality  property, hence, it is specified among other invariant, by its \emph{fundamental class}, denoted $\omega$. The degree of $\omega$, the maximal one, is called the \emph{formal dimension} of $X$. 
If only $\dim V$ is finite, referring to \cite{Mu94}, $X$ is in a larger class than that of Poincaré spaces, namely the class of  \emph{Gorenstein spaces}  (cf. \S 2 below for more details). One of the main ingredients we will use is  the {\it Milnor-Moore spectral sequence} \eqref{Milnor Moore spec} (see \S 2). This provides an algebraic definition of a good lower bound of the rational LS-category $cat_0(X):=cat(X_0)$, namely, the {\it rational Toomer invariant} $e_0(X)$ as follows:
 \begin{equation}\label{e_0}
e_0(X) = max\{m\in \mathbb{N}\; | \; E_{\infty}^{m,*}(\Lambda V)\neq 0\}
 \end{equation}
 with $E_{\infty}^{m,*}(\Lambda V)$, the infinity term of \eqref{Milnor Moore spec}.
   Referring to \cite{FHL}, $cat_0(X)=e_0(X)$ for every elliptic space $X$.

Focusing on the differential, clearly its lower (homogeneous part) $d_k$ is also a differential so that $(\Lambda V,d_k)$ is also a commutative differential graded algebra (cdga for short). 
If  $(\Lambda V,d_k)$ is elliptic,  by the convergence of \eqref{Milnor Moore spec}, so is $(\Lambda V,d)$. Moreover, in such a case, we have  \cite[Proposition 3]{L-M02}:
\begin{equation}\label{finite-cace}
e_0(X) := e_0(\Lambda V,d) = e_0(\Lambda V,d_k) = \dim V^{\text{odd}} + (k-2)\dim V^{\text{even}}.
\end{equation}
 
 The main goal of this work is to give an explicit algorithm to determine  $e_0(\Lambda V,d)$ when $(\Lambda V,d_k)$ is not necessarily elliptic, hence an exact value of $cat_0(X)$ for every arbitrary rationally elliptic space  $X$. 
 
 Thereafter, $\dim V$ is finite and, while $(\Lambda V,d_k)$ is not necessarily elliptic, we still suppose that  $(\Lambda V,d)$ is elliptic. Hence,  referring to  \cite{Mu94} we know that $(\Lambda V,d_k)$ and  $(\Lambda V,d)$ are {Gorenstein algebras} with the same {\it formal dimension} $N=\max\{r\; | H^r(\Lambda V,d_k)\neq0\}$. To treat this large class, we make use of  the convergence Eilenberg-Moore spectral sequence \eqref{Ext Milnor Moore spec} (see below). This  was firstly introduced in \cite{Mu94} and subsequently used in \cite{R2} to introduce, in the same spirit of $e_0(X)$, a {\it new lower bound for} $cat(X)$ which we denote $\mathrm{R}_0(X)$ and will use considerably below. 
 
 Using once more the hypothesis $\dim V < \infty$, we see easily that $H^N(\Lambda V,d_k)$ is finite dimensional. By a spectral sequence approach
combined with an algorithm à la Lechuga-Murillo \cite{L-M02}, the third and fourth authors showed in \cite{B-R} that one class $\omega_0 \in H^N(\Lambda V,d_k)$ survives to the infinity term of \eqref{Milnor Moore spec} and induces the fundamental class $\omega$ of $(\Lambda V,d)$. The Toomer invariant of $(\Lambda V,d)$ is then given explicitly by:
   \begin{equation}\label{e_0(omega)}
   e_0(\Lambda V,d) = e_0(\omega) = sup\{j\mid\exists \alpha\text{ such that } \omega =[\alpha]\text{ and }  \alpha \in (\Lambda ^{\geq j}V)^N\}.
   \end{equation}
Another cdga associated to $(\Lambda V,d)$ is the \emph{pure associated model} $(\Lambda V,d_{\sigma})$, \cite{FHTb, H77}. A spectacular result established by S. Halperin in  \cite{H77} states that $(\Lambda V,d)$ is elliptic if and only if $(\Lambda V,d_{\sigma})$ is. When $\dim V$ is finite,  $(\Lambda V,d_{\sigma})$ is also a Gorenstein algebra with  the same formal dimension as $(\Lambda V,d)$. 
  
 
 In  \S 3 we give an algorithm to construct a generating class $[f_{\sigma}]$ of $Ext^N_{(\Lambda V,d_{\sigma})}(\mathbb{Q},(\Lambda V,d_{\sigma}))$. 
 In  \S 4, by identifying the $E_{\infty}$ term of the convergent spectral sequence \eqref{Ext Milnor Moore spec} 
 with   $Ext^N_{(\Lambda V,d)}(\mathbb{Q},(\Lambda V,d))$, we obtain, through an algorithm à la Lechuga-Murillo, a generating class $[f_k^t]$ of this latter from   the generating class  $[f_k]$ of $Ext^N_{(\Lambda V,d_k)}(\mathbb{Q},(\Lambda V,d_k))$. 
 Finally, in  \S 5 we use the spectral sequence \eqref{second spectral sequence} (see below),  introduced  in \cite{R1}, to build a generating class $[f_{\sigma}^l]$ of $Ext^N_{(\Lambda V,d)}(\mathbb{Q},(\Lambda V,d))$ from $[f_{\sigma}]$.
 
  By way of a summary, given an elliptic space $X$ with Sullivan minimal model $(\Lambda V,d)$, by applying the algorithm from in \S 3 to $(\Lambda V,(d_{k})_{\sigma})$ we obtain    the generating class, denoted $[f_{k,\sigma}]$, of $Ext^{N}_{(\Lambda V,(d_{k})_{\sigma})}(\mathbb{K},(\Lambda V,(d_{k})_{\sigma}))$. Then, noticing that $(d_{k})_{\sigma}= (d_{\sigma})_{k}$, we apply  the  algorithm developed in \S 4 to $(\Lambda V, (d_{\sigma})_{k})$
    to get the generating class $[f_{k,\sigma}^t]$ of $Ext^{N}_{(\Lambda V,d_{\sigma})}(\mathbb{K},(\Lambda V,d_{\sigma}))$. We then use the   algorithm described in \S 5 to $(\Lambda V, d_{\sigma})$ to obtain  the generating class $[(f_{k,\sigma}^ t)^l]$ of $Ext^{N}_{(\Lambda V,d)}(\mathbb{K},(\Lambda V,d)$. 
    
    As a summary, our algorithm complements those developed successively in \cite{Mu}, \cite{L-M02} and \cite{B-R}. Indeed, a combination of their contributions makes it possible to determine $e_0$ and therefore $cat_0$ in a quasi-complete way. The missing point is that this way does not allow us to know, when $H^*(\Lambda V,d_{k, \sigma})$ does not verify the Poincaré property, which of the generators of $H^N(\Lambda V,d_{k, \sigma})$ persists up to the infinite level of the spectral sequence \eqref{Milnor Moore spec}. Our work contributes to the extension of subsequent algorithms to the level of $Ext$ and thus takes advantage of the fact that $(\Lambda V,d_{k, \sigma})$ is a Gorenstein algebra.

With the notations above, our main result reads as follows:
  \begin{theorem}\label{th1.1}
  Let $X$ be a rationally elliptic space. Then, its fundamental class is given by  $[(f_{k,\sigma}^t)^l(1)]$. Hence, $cat_0(X) = e_0(X) = e_0([(f_{k,\sigma}^t)^l(1)]).$
  \end{theorem}
Note that in \cite[Proposition 20]{Le02}, L. Lechuga constructed an algorithm to compute the rational category of any elliptic space under the condition $(d - d_{\sigma})V \subseteq \Lambda^{>l_{\sigma}}V$ where $l_{\sigma}=l_0(\Lambda V,d_{\sigma})$. This algorithm is based on Gr\"oebner basis calculus, which computes $l_{\sigma}$. Our result is therefore a generalization of Lechuga's. Moreover, using the algorithm given  in \S 4, 
we obtain an explicit formula for $l_0(X)$ as follows:
\begin{theorem}\label{th1.2}(Theorem \ref{L-inv})
Let $X$ be a rationally elliptic space. Then, $l_0(X)=L_0(X) = t-1$.
\end{theorem}
This in particular gives a method to compute $l_0(\Lambda V,d)$ for any cdga $(\Lambda V,d)$ which is in fact alternative and more general than the one given in \cite{Le02}.
Recall that  $L_0(X)$ is the {\it Ext-Ginsburg invariant} introduced recently in \cite{A-R}  in the spirit of the  Ginsburg's classical invariant $l_0(X)$.

Using this coincidence between $l_0(X)$ and $L_0(X)$ we give an alternative and explicit proof of  \cite[Theorem 1]{A-R} which improves \cite[Theorem 1]{R2} as follows:
\begin{theorem}\label{th1.3} (Theorem \ref{R-inv})
Let $X$ be a rationally elliptic space. Then, $\mathrm{R}_0(X) = e_0(X) = cat_0(X)$.
\end{theorem}
  
 We present the proofs of Theorems \ref{th1.2} and \ref{th1.3} in \S 4, while the proof of Theorem \ref{th1.1} is reported in \S 5. Section \S 6, entitled "Concluding remarks", is devoted to some comments on our main result as well as its interaction with previous results on the subject.

\section{Preliminaries}
In this section, we review essential tools used in this work. Unless otherwise stated, all algebraic structures are defined over a field $\mathbb{K}$.
\subsection{Sullivan Models}
In this subsection, we assume $\mathbb{K}$ of characteristic zero. The standard reference for this section is \cite{FHTb}.

Let  $V=\oplus_{i\geq 1}V^{i}$ be a graded vector space over $\mathbb{K}$. The free  commutative graded algebra over $V$, denoted $\Lambda V$, is the quotient of the tensor graded algebra $TV$ by the graded ideal $I$ generated by homogeneous elements of the form $x\otimes y-(-1)^{| x | | y |}y\otimes x$; $x, y \in V$ where $|z|$ denotes the degree of $z\in V$. Thus,
$\Lambda V=Exterior(V^{\text{odd}})\otimes\,Sym(V^{\text{even}})=\bigoplus_{i\geq 0}\Lambda ^{i} V$
with $\Lambda ^{i} V$ denoting the linear span of the elements $v_{1}^{n_{1}}\ldots v_{r}^{n_{r}}$; $n_{1}+\cdots+n_{r}=i$ and $n_{i}=1$ if $v_{i}\in V^{\text{odd}}$.

  A \emph{Sullivan algebra} is a commutative cochain algebra  $(\Lambda V,d)$ generated by the graded vector space $V=\oplus_{i\geq 1}V^{i}$ that is the union of an increasing sequence of graded subspaces $\{V(k)\}_{k\geq 0}$ such that $d=0$ on $V(0)$ and $d: V(k)\rightarrow \Lambda V(k-1),\; k\geq 1$. In other words, $d$ preserves each  $\Lambda V(k)$ and there exists a subspace $V_k\subseteq \Lambda V(k)$ such that $\Lambda V(k) = \Lambda V(k-1)\otimes V_k$ and $d: V_k\rightarrow \Lambda V(k-1),\; k\geq 1$. 
  Such an algebra is called \emph{minimal} if moreover:
  $$Im(d) \subseteq \Lambda ^+V\otimes \Lambda ^+V.$$
  If $(\Lambda V,d)$ is $1$-connected, that is $V^{1}=0$, then  $(\Lambda V,d)$ is minimal if and only if 
  $$Im(d)\subset \Lambda^{\geq 2} V=\oplus_{i\geq 2}\Lambda ^{i} V.$$

  A \emph{Sullivan model} for a cdga $(A,d)$ is a quasi-isomorphism (i.e. a morphism inducing an isomorphism in cohomology):
 \begin{center}
  $ m:\, (\Lambda V ,d)\stackrel{\simeq}{\rightarrow}(A,d)$
 \end{center}
 from a Sullivan algebra $(\Lambda V ,d)$. Often we simply say that $(\Lambda V,d)$ is a model  of $(A,d)$.
 
 Referring to \cite[ Proposition 12.2]{FHTb} we know that any commutative  \emph{cohomologically $1$-connected} cdga $(A,d)$, in the sense that $H^{0}(A)=\mathbb{K}$ and $H^{1}(A)=0$,  has minimal Sullivan model.
 
 If $X$ is a path-connected topological space, a Sullivan model
 \begin{center}
   $ m:\, (\Lambda V ,d)\stackrel{\simeq}{\rightarrow} C^{*}(X,\mathbb{K})$
 \end{center}
 of  the cochains $C^{*}(X,\mathbb{K})$ on $X$, is called a \emph{Sullivan model for $X$}. In particular, if $X$ is $1$-connected, then it admits a minimal Sullivan model.

A \emph{KS-extension} of  an augmented cdga $\epsilon:\, (A,d)\rightarrow (\mathbb{K},0)$ (i.e. a cdga with the algebra morphism $\epsilon$ sending the ideal of augmentation $A^+= A^{>0}$ to $0$) is an exact sequence of cdga morphisms
 $$(A,d_A)\rightarrow (A\otimes \Lambda V,d_{A\otimes \Lambda V} )\rightarrow (\Lambda V,\overline{d} )$$ where, $d_{A\otimes \Lambda V}$ restricts to $d_A$ on $A$, $dv_{i}\in A\otimes \Lambda V_{<i}$, and  $(\Lambda V,\bar{d})$ is a Sullivan algebra with differential  $\bar{d}$ denoting the obvious quotient differential induced by $d_{A\otimes \Lambda V}$ on $V$.
 
 \subsection{Differential Ext}
 In what follows, we fix $(R,d)$ as a cdga over $\mathbb{K}$.
 
A left $(R,d)$-module $(M,d)$ is said to be \emph{semifree} if $M$ is the union of an increasing sequence $\{M(k)\}_{k\geq 0}$
of $(R,d)$-submodules such that $M(0)$ and each of the quotients $M(k)/M(k-1)$ are $R$-free modules on a basis of cycles.
Such an increasing sequence is called a \emph{semifree filtration} of $(M,d)$.

A {semifree resolution} of an $(R,d)$-module $(A,d)$ is an $(R,d)$-semifree module $(M,d)$ together with a quasi-isomorphism
\begin{center}
 $ m:\, (M,d)\stackrel{\simeq}{\rightarrow}(A,d)$
\end{center}
of $(R,d)$-modules. Recall from \cite[Proposition 6.6]{FHTb} that:
  \begin{enumerate}
    \item  Every $(R,d)$ module $(A,d)$ has a semifree resolution $m:\, (M,d)\stackrel{\simeq}{\rightarrow}(A,d)$.
    \item If $m':\, (M',d)\stackrel{\simeq}{\rightarrow}(A,d)$ is a second semifree resolution then there is an equivalence of $(R,d)$-modules $\alpha: (M',d)\rightarrow (M,d)$ such that $m \circ \alpha\sim m'$.
  \end{enumerate}
  
\begin{remark}\label{form}
  Let $\{M(k)\}$ be a semifree filtration of $(M,d)$. Then $M(0)$ and each $M(k)/M(k-1)$ have the form $(R,d)\otimes (Z(k),0)$ where $Z(k)$ is a free $\mathbb{K}$-module and so that each surjection $M(k) \rightarrow R\otimes Z(k)$ splits as follows:
  \begin{center}
    $M(k)=M(k-1)\oplus( R\otimes Z(k))$, and $d:\,Z(k)\rightarrow M(k-1)$.
  \end{center}
 Thus, forgetting the differentials, we conclude that $M=R\otimes(\bigoplus_{k=0}^{\infty})$ is a free $R$-module.
\end{remark}

Let $\eta:\, (P,d)\rightarrow (Q,d)$ be a morphism of $(R,d)$-modules, and for any third $(R,d)$-module $(M,d)$, denote by
$$
\begin{array}{lccl}
  Hom_{R}(M,\eta): &  Hom_{R}(M,P) & \rightarrow &  Hom_{R}(M,Q)\\
                   &    \varphi    &   \mapsto   &  \eta \circ\varphi
\end{array}
$$
the morphism of complexes induced by $\eta$.
Referring to \cite[Proposition 6.4]{FHTb}, if $(M,d)$ is a semifree $(R,d)$-module and $\eta$ is a quasi-isomorphism, then
  \begin{itemize}
    \item $Hom_{R}(M,\eta)$ is a quasi-isomorphism.
    \item Given a diagram of morphism of $(R,d)$-modules,
  \begin{equation}
  \begin{array}{c c c}
  & & (P,d)\\
  & & \downarrow \simeq \eta\\
  (M,d) & \stackrel{\psi}{\rightarrow} & (Q,d)
  \end{array}
  \end{equation}  

there is a unique homotopy class of morphisms $\phi:\,(M,d) \rightarrow (P,d)$ such that $\eta \circ \phi \sim \psi$.
\item A quasi-isomorphism between semifree $(R,d)$ modules is an equivalence.
  \end{itemize}

The second property is called \emph{the lifting lemma}.

Given $f:\,M\rightarrow M'$ and $g:\,P'\rightarrow P$ two $R$-linear morphisms between left $(R,d)$-modules and $h:\, Q\rightarrow Q'$  an $R$-linear morphism between right $(R,d)$-modules. We  define

\begin{center}
  $Hom_{R}(f,g):\, Hom_{R}(M',P')\rightarrow Hom_{R}(M,P)$ by $\phi \longmapsto (-1)^{\deg f(\deg g+\deg\phi)}g \circ \phi \circ f.$
\end{center}
and
\begin{center}
  $h\otimes_{R}f:\, Q\otimes_{R}M\rightarrow Q'\otimes_{R}M'$ by $q\otimes_{R}m\longmapsto (-1)^{\deg f \deg q}h(q)\otimes_{R}f(m)$.
\end{center}
If moreover $f$, $g$ and $h$ commute with the differentials then $Hom_{R}(f,g)$ and $h\otimes_{R}f$ are morphisms of graded chain complexes.
Furthermore, if $(M,d)$ and $(M',d)$ are $(R,d)$-semifree, then \cite[Proposition 6.7]{FHTb} the following holds:
  \begin{enumerate}
    \item If $f$ and $g$ are quasi-isomorphisms then so is $Hom_{R}(f,g)$.
    \item If $f$ and $h$ are quasi-isomorphisms then so is $h\otimes_{R}f$.
  \end{enumerate}

Given two $(A,d)$-modules $(M,d)$ and $(N,d)$ and  $(P,d)$ a semifree resolution of $(M,d)$. The \emph{Eilenberg-Moore generalized functor} called here  the \emph{Ext differential}  is defined as follows:
$$
  Ext_{(A,d)}((M,d),(N,d))=H(Hom_{A}(P,N), D).
$$
It follows from the previous properties that this functor is independent of the choice of the semifree resolution of $(M,d)$. More precisely
$$Hom_{A}(P,N)=\bigoplus_{p\geq 0}Hom_{A}^{p,*}(P,N)=\bigoplus_{p\geq 0}\bigoplus_{q\geq 0}Hom_{A}^{p,q}(P,N)$$
where   
$Hom_{A}^{p,q}(P,N)=Hom_{A}(P^{q},N^{p+q})$
 has the differential $D$ defined for any  $f: P^q\rightarrow N^{p+q}$, of degree $p$, by 
\begin{equation}\label{differential}
D(f)=d_N\circ f+(-1)^{p+1}f\circ d_P.
\end{equation}
So, we obtain the   graduation:
$$Ext_{A}(M,N) = \oplus_j[Ext_{A}(M,N)]^j= \oplus_j [\oplus_{p} Ext_{A}^{p,j-p}(M,N)] = \oplus_j [\oplus_{p}H^{p,j-p}(Hom_{A}(P,N))]$$
where $H^{p,j-p}(Hom_{A}(P,N))$ stands for the cohomology at the middle term in the following complex:
$$ \cdots \stackrel{D}{\rightarrow} Hom_{A}(P^{j-p},N^{j-1})\stackrel{D}{\rightarrow} Hom_{A}(P^{j-p},N^{j})\stackrel{D}{\rightarrow} Hom_{A}(P^{j-p},N^{j+1})\stackrel{D}{\rightarrow} \cdots .$$
$p$ is called \emph{the cohomological degree}, $q=j-p$ \emph{the internal or the complementary degree} and $j$  \emph{the total degree}.
\subsection{Gorenstein spaces and the evaluation map}

In this subsection we consider cohomologically $1$-connected cdga's and  $1$-connected spaces.

An augmented cdga $(A,d)$ is called a \emph{Gorenstein dga} or simply a \emph{Gorenstein algebra} if the graded $\mathbb{K}$-vector space $Ext_{A}(\mathbb{K},A)$ is one-dimensional. The generating class $\Omega$ of $Ext_{A}(\mathbb{K},A)$ and its degree $|\Omega|$, denoted $fd(A,d)$, are called the \emph{fundamental class} and \emph{formal dimension} of $(A,d), respectively$. 

  A pointed topological space $X$ is said to be a \emph{Gorenstein space} over $\mathbb{K}$ (with formal dimension denoted $fd(X)$) if $C^{*}(X,\mathbb{K})$ is a  Gorenstein algebra.
Referring to \cite{FHT88}, given a quasi-isomorphism $A\stackrel{\simeq}{\rightarrow}B$ of augmented cdga's, we obtain the following chain of isomorphisms
\begin{equation}\label{evv}
  Ext_{(A,d)}(\mathbb{K},(A,d))\stackrel{\cong}\rightarrow Ext_{(A,d)}(\mathbb{K},(B,d))\stackrel{\cong}\longleftarrow Ext_{(B,d)}(\mathbb{K},(B,d)).
\end{equation}

This composition will be viewed as an identification, that is:
$$
Ext_{(A,d)}(\mathbb{K},(A,d))= Ext_{(B,d)}(\mathbb{K},(B,d)).
$$
Thus, if  $(\Lambda V,d)$ is a Sullivan model of $(A,d)$ (resp. of $X$), the latter   is Gorenstein over $\mathbb{K}$ if and only if $(\Lambda V,d)$ is.
This holds when $V$ is finite dimensional \cite[Theorem A]{Mu94}. If moreover, $\dim H^{*}(\Lambda V,d)<\infty$, then $H^{*}(\Lambda V,d)$ satisfies Poincaré duality property over $\mathbb{K}$ and $fd(\Lambda V,d):=N$ is exactly the degree of its fundamental class $[\omega]$; the unique generating element of $H^N(\Lambda V, d)$ \cite[Theorem A]{Mu94}. Recall that a graded algebra $A$ of finite dimension is called of \emph{formal dimension} $m$ if $A^m\neq0$ and $A^{>m} =0$. Such an algebra is said a \emph{Poincaré duality algebra} or it satisfies \emph{Poincaré duality property} over $\mathbb{K}$ if it is commutative and satisfies the following conditions:
	\begin{enumerate}
	\item $A^m \cong \mathbb{K}\omega$.
	\item $\forall p, A^p\otimes A^{m-p} \rightarrow A^m \cong \mathbb{K}$ is a nondegenerate bilinear form.
	\end{enumerate} 
	The generating element $\omega$ of degree $m$ is called the \emph{fundamental class} of $A$.

Let $(A,d)$ be an augmented cdga and consider $\mathbb{K}$ as a cdga with  differential $d_{\mathbb{K}}=0$. The chain evaluation map 
\begin{equation}
\begin{array}{cccc}\label{cevmap}
cev: &  Hom_{(A,d)}((P,d),(A,d)) &  \rightarrow  &  (A,d)\\
    &       f  &  \mapsto  &   f(p)
\end{array}
\end{equation}
 induces in homology, 
 the \emph{evaluation map} \cite{FHT88, Mu2} is the linear map
\begin{equation}\label{evmap}
\begin{array}{cccc}
ev_{(A,d)}: &  Ext_{(A,d)}(\mathbb{K},(A,d))  & \rightarrow  &  H(A,d)\\
        & [f]  & \mapsto &  [f(p)]
\end{array}
\end{equation}
Here $[f]$ is a cohomology class  represented by a cocycle $f: (P,d) \rightarrow (A,d)$, where $(P,d)\stackrel{\simeq}{\rightarrow} (\mathbb{K},0)$ is an $(A,d)$-semifree resolution of $\mathbb{K}$, and $p$ is a cocycle in $(P,d)$ representing $1_{\mathbb{K}}$.
Again from (\ref{evv}) we deduce that $ev_{A}$ is preserved by  quasi-isomorphisms.

The evaluation map $ev_{X}$ of a pointed topological space $X$ is by definition that  of $C^{*}(X;\mathbb{K})$. Thus, if $(\Lambda V,d)$ is a minimal Sullivan model of $X$, we will make the following identification in what follows:
$$ev_X = ev_{(\Lambda V,d)}:  Ext_{(\Lambda V,d)}(\mathbb{K},(\Lambda V,d))  \rightarrow  H(\Lambda V,d).$$

Referring to \cite[Theorem A]{Mu2}, when $\dim V<\infty$ we have:
 $$dimH^{*}(\Lambda V,d)<\infty \Leftrightarrow ev_{(\Lambda V,d)}\neq0$$ or equivalently, any Gorenstein algebra satisfies Poincaré duality property if and only if it has a non-zero evaluation map.
 
\subsection{Spectral sequences}
A cohomology \emph{spectral sequence} (see \cite{Mc, FHTb}) is a sequence of differential bi-graded complexes $(E_{r}^{*,*},d_{r},\sigma_{r})$, $r\geq s$ (some integer $s$) over $\mathbb{K}$  with $E_{r}^{*,*}=\{E_{r}^{p,q}\}$, $d_{r}$ is a differential of bi-degree $(r,-r+1)$ and  $\sigma_{r}:\, H(E_{r}^{*,*})\stackrel{\cong}{\rightarrow} E_{r+1}^{*,*}$ is an isomorphism of bi-graded complexes.

The spectral sequence $(E_{r},d_{r})$ is called \emph{convergent} if for each $(p,q)$ there is an integer $r(p,q)$ such that $$ E^{p,q}_{n} = E^{p,q}_{r(p,q)},\;\;  \forall n\geq r(p,q).$$
In this case the $\infty$ term of the spectral sequence is the bi-graded complex $E^{*,*}_{\infty} =\oplus_{p,q} E^{p,q}_{\infty}$ where 
$$E^{p,q}_{\infty} = E^{p,q}_{r(p,q)},\;\;  \forall (p,q).$$
Given a cdga $(A,d)$. A filtration $FA$ on $(A,d)$ is a family of graded subalgebras $\{F^{p}A\}$ for $p\in \mathbb{Z}$ so that
$$\cdots \subseteq F^{p+1}A\subseteq F^{p}A \subseteq F^{p-1}A\subseteq \cdots \subseteq A$$ and $d(F^{p}A) \subseteq F^{p}A$ for all $p\in \mathbb{Z}$. In particular,  $d$ induces a differential on each $F^pA$ so that $(F^pA,d_{F^pA})$ becomes also a differential graded algebra.
The filtered cdga $(A, FA, d)$ determines the associated differential bi-graded algebra $(GA,Gd)$ where
$$G^{p,q}A=(F^{p}A/F^{p+1}A)^{p+q}$$
and $Gd$ is the differential naturally induced by $d$. Here $p$ is called the \emph{filtration degree}, $q$ is the \emph{complementary degree} and $p+q$ is the \emph{total degree}.

Now, denote (for each integer $r\geq 0$) by $Z_{r}^{p}$ and $B_{r}^{p}$ the sub-algebras
$$ Z_{r}^{p,q} = \{ x\in [F^pA]^{p+q} | dx\in [F^{p+r}(A)]^{p+q+1}\}$$
and
$$B_{r-1}^{p,q} = d([F^{p-r+1}A]^{p+q-1})\cap F^pA =  d(Z_{r-1}^{p-r+1,q+r-2}).$$
Clearly,  $B_{0}^{p}\subset B_{1}^{p}\subseteq \cdots B_{r}^{p}\subseteq \cdots \subseteq Z_{r}^{p}\subseteq \cdots \subseteq Z_{0}^{p}.$
The $r$-term ($r\geq 0$) of this spectral sequence is the
bi-graded algebra $E_{r}=\{E_{r}^{p,q}\}$ where

\begin{equation}\label{Ekterm}
E_{r}^{p,q}=(Z_{r}^{p}/(Z_{r-1}^{p+1}+B_{r-1}^{p}))^{p+q}=Z_{r}^{p,q}/Z_{r-1}^{p+1,q-1} +  B_{r-1}^{p,q}.
\end{equation}

It follows from the definition that $(E_{0},d_{0})=(GA,Gd)$ and  $d$ factors to define a differential $d_{r}$ in $E_{r}$ of bi-degree $(r,1-r)$. Moreover, the inclusion $Z_{r+1}^{p}\hookrightarrow Z_{r}^{p}$ induces an isomorphism of bi-graded algebras $E_{r+1}\stackrel{\cong}{\rightarrow} H(E_{r})$ and
the filtration $\{F^{p}A\}$ on  $A$ induces a filtration on $H(A,d)$ as follows:
$$F^{p}H(A,d)=Im(H(F^{P}A)\rightarrow H(A)).$$
If the associated spectral sequence is convergent and there is an isomorphism $E_{\infty}\cong GH(A)$ of bi-graded algebras then, we say that it   converges to $H(A)$. This is expressed (quite often) by:
\begin{equation}
   E_2^{p,q} \implies  H^{p+q}(A,d).
\end{equation}

\subsubsection{Milnor-Moore and Eilenberg-Moore spectral sequences}
In this subsection we consider again cohomologically $1$-connected cdga algebras $1$-connected spaces.

Given a minimal Sullivan algebra $(\Lambda V,d)$ with $d=d_{k}+d_{k+1}+\cdots$ (some $k\geq 2$) where $d_{i}(V)\subset \Lambda^{\geq i} V$ ($\forall i\geq k$). Notice that $d_k$ is indeed a differential. We endow $\Lambda V$ with the decreasing filtration:
\begin{equation}\label{fil-MM-spec}
F^{p}\Lambda V=\Lambda^{\geq p} V,\; p\geq 0.
\end{equation}
It is straightforward to verify that $(F^{p}\Lambda V)$ is a decreasing sequence and $dF^{p}\Lambda V \subset F^{p}\Lambda V$. Consequently, we obtain the so-called (convergent)
  \emph{Milnor-Moore spectral sequence}:
\begin{equation}\label{Milnor Moore spec}
    E_k^{p,q}(\Lambda V) = H^{p,q}(\Lambda V,d_{k})\implies  H^{p+q}(\Lambda V,d).
  \end{equation}
    
Next,(see \cite{Mu94}, \cite{R2}), on  the cdga $A=(Hom_{\Lambda V}((\Lambda V\otimes \Lambda sV,d),\Lambda V),D)$, where  \cite{FHTb} $(\Lambda V\otimes \Lambda sV,d)$  stands for the $(\Lambda V, d)$-semifree resolution of $\mathbb{K}$ with differential $d$ extending that of $\Lambda V$ by $dsv=v-sdv$, we
consider the filtration
\begin{equation}\label{filt-Ext-MM spec}
F^{p}A=\{f\in Hom_{\Lambda V}((\Lambda V\otimes\Lambda sV,d),\Lambda V)\; |\;  f(\Lambda sV)\subset\Lambda^{\geq p} V\},\; p\geq 0.
\end{equation}
Notice that since $f$ is a $(\Lambda V, d)$-morphism, then, $f\in F^pA$ if and only if $f(\Lambda V \otimes \Lambda sV)\subset\Lambda^{\geq p} V$, that is $f(1)\subset\Lambda^{\geq p} V$ and $f(\Lambda sV)\subset\Lambda^{\geq p} V$.
Here again, it is straightforward to see that $(F^{p}A)$ is a decreasing sequence and $dF^{p}A \subset F^{p}A$.
Then, we get the \emph{Eilenberg-Moore spectral sequence} \cite{Mu94} and \cite{R2}:
\begin{equation}\label{Ext Milnor Moore spec}
E_k^{p,q}(A) = Ext^{p,q}_{(\Lambda V,d_{k})}(\mathbb{K},(\Lambda V,d_{k})) \implies  Ext^{p+q}_{(\Lambda V,d)}(\mathbb{K},(\Lambda V,d)).
\end{equation}
 
Recall that a Sullivan algebra is called elliptic if $\dim V$ and $\dim H(\Lambda V,d)$ are both finite dimensional. It results, from the convergence of these spectral sequences that if $(\Lambda V, d_{k})$  is an elliptic (resp. a Gorenstein) Sullivan algebra, then $(\Lambda V, d)$ is.

\subsubsection{Odd and odd-Eilenberg-Moore spectral sequences}
 To give a convenient characterization of ellipticity, S. Halperin associated \cite{H77} to an arbitrary  Sullivan algebra $(\Lambda V,d)$ whose generating space $V$ is finite dimensional, another  Sullivan algebra denoted $(\Lambda V,d_{\sigma})$ with differential $d_{\sigma}$ satisfying:
  $$d_{\sigma}(V^{\text{even}})=0,\; \hbox{and}\; (d-d_{\sigma})(V^{\text{odd}})\subseteq \Lambda V^{\text{even}}\oplus \Lambda^+V^{\text{odd}}
 .$$
  $(\Lambda V,d_{\sigma})$ is called, the \emph{associated pure algebra} of $(\Lambda V,d)$. Notice that the second property is equivalent to $d_{\sigma}(V^{\text{odd}})\subseteq \Lambda V^{\text{even}}$.
 He then introduced the following spectral sequence 
 \begin{equation}\label{odd spec}
   E_k^{p,q}(\Lambda V)(pure) = H^{p,q}(\Lambda V,d_{\sigma})\implies  H^{p+q}(\Lambda V,d)
 \end{equation} which connects  $(\Lambda V,d)$ and $(\Lambda V,d_{\sigma})$ called
\emph{the odd spectral sequence}. Its main result states that $\dim H(\Lambda V,d) < \infty \Leftrightarrow \dim H(\Lambda V,d_{\sigma}) < \infty.$
Thus, $(\Lambda V,d)$ is elliptic if and only if $(\Lambda V,d_{\sigma})$ is.

Another spectral sequence, introduced by the fourth author in \cite{R1} is called \emph{the odd-Eilenberg-Moore or Ext-odd
 spectral sequence}:
\begin{equation}\label{Ext-odd spec}
E_k^{p,q}(A)(pure) = Ext^{p,q}_{(\Lambda V, d_{\sigma})}(\mathbb{K},(\Lambda V, d_{\sigma})) \implies  Ext^{p+q}_{(\Lambda V, d)}(\mathbb{K},(\Lambda V, d)).
\end{equation}
More details about it will be  discussed in \S 5.

\section{A basis of $Ext_{(\Lambda V, d_{\sigma})}(\mathbb{K},(\Lambda V, d_{\sigma}))$.}

In this section we assume $\mathbb{K}=\mathbb{Q}$. 
Recall that a Sullivan algebra $(\Lambda V,d)$ is said to be \emph{pure}  if $\dim V < \infty$, $dV^{\text{even}}=0$ and  $dV^{\text{odd}}\subseteq  \Lambda V^{\text{even}}$ \cite{FHTb, H77}. If moreover $\dim H(\Lambda V, d)<\infty$ we say that $(\Lambda V, d)$ is an \emph{elliptic pure Sullivan algebra}. Any space whose minimal model is a pure (elliptic) model is called a \emph{pure (elliptic) space}.

We assume that $(\Lambda V, d)$ is a minimal pure  Sullivan algebra, that is $\dim V < \infty$ and its differential $d$ satisfies $d= \sum_k d_{\geq k}$ (for some fixed $k\geq 2$),  $dV^{\text{even}}=0$ and  $dV^{\text{odd}}\subseteq  \Lambda V^{\text{even}}$. 

Let $(x_{1},\ldots,x_{n})$ be a basis of $V^{\text{even}}$  and $(y_{1}, \ldots ,y_{m})$  be a basis of  $V^{\text{odd}}$. Thus,
$$\Lambda V=\mathbb{Q}[x_{1}, \ldots ,x_{n}]\otimes \Lambda(y_{1}, \ldots ,y_{m}),$$
	with $dx_{i}=0$ for all $1\leq i\leq n$ and $dy_{j}\in\mathbb{Q}[x_{1}, \ldots ,x_{n}]$ for all $1\leq j\leq m$.
 Its formal dimension is given by the formula \cite[Proposition 5.2.]{FHT88} 
 \begin{equation}\label{fd}
 fd(\Lambda V,d) := N =\sum_{j=1}^{m}| y_{j}|+\sum_{i=1}^{n}(1-| x_{i}|).
 \end{equation}

A $(\Lambda V, d)$-semifree resolution of $(\mathbb{Q},0)$, called an acyclic closure of $(\Lambda V, d)$ (see \cite{FHTb}), is quasi-isomorphic to the cdga $(\Lambda V\otimes(\mathbb{K}\oplus sV), d)$ with differential (denoted also $d$) extending that of  $(\Lambda V,d)$ as follows :
	 \begin{equation}\label{d(sv)}
	  d(sv)=v - s(dv),\; \hbox{for any} \; v\in V.
	 \end{equation}	
Therefore, being  $\dim V < \infty$, $(\Lambda V, d)$ is a Gorenstein algebra so, 	there exists a unique generating class $\Omega = [f]$ of $Ext_{(\Lambda V, d)}(\mathbb{Q},(\Lambda V, d))$ represented by a  non-degenerate cocycle morphism
	$f\in Hom_{(\Lambda V,d)}((\Lambda V\otimes(\mathbb{K}\oplus sV), d) , (\Lambda V, d))$.
In particular, we have  $D(f) =  d\circ f +  (-1)^{|f|+1} f\circ d  = 0$ (cf. (\ref{differential})), that is
\begin{equation}\label{D}
d\circ f = (-1)^{|f|} f\circ d.
	\end{equation}	

The main result of this section is :

\begin{proposition}\label{Basis1}
Let $(\Lambda V,d)$ be a  pure (not necessarily elliptic) algebra. With the notations above,
    a basis class $[f]$  of $Ext_{(\Lambda V, d)}(\mathbb{Q},(\Lambda V, d))$ is completely determined by $f(\omega)$ and $f(sx_{i})$ for all  $x_i$ in $V^{\text{even}}$.
\end{proposition}

\begin{proof}
	To prove this proposition, we consider two cases based on the evaluation map.
	\begin{itemize}
		\item 
We begin by considering the case where $ev_{\Lambda V}\neq0$ so that  $(\Lambda V,d)$ is elliptic \cite[Theorem A]{Mu2}, hence, a Poincaré duality algebra with a non-zero fundamental class $\omega$. Since $\Lambda V\otimes(\mathbb{K}\oplus sV)$ is positively graded, we may put $p=1_{\mathbb{Q}}:=1$ so that $\omega = ev_{(\Lambda V,d)}:= [f(1)]$.

Recall that an explicit expression of a cocycle representing $\omega$ is given for instance in \cite{Mu}.
Then, since
 $f$  is a  $(\Lambda V,d)$-module, it remains to determine it on $\Lambda sV$. For this, 
 since $(\Lambda V, d)$ is pure, formulas (\ref{d(sv)}) and (\ref{D}) above applied to $sV^{\text{even}}$ and $sV^{\text{odd}}$ reduce, respectively, to : 
\begin{equation}\label{sv}
  d(f(sx_{i}))= (-1)^{N}f(x_{i}) = (-1)^{N}x_i\omega \text{ and } 
  d(f(sy_{j})) = (-1)^{N}y_{j}\omega -(-1)^{N}f(sdy_{j}).
\end{equation}

As $|x_i|\geq 2$, we have $|f(x_{i})|>N$, so that   $d(f(x_{i}))= f(d(x_i))=0$ which implies that  $f(x_{i})=x_i\omega$ is a coboundary. Consequently, the first obstruction is to find $T_{i}\in \Lambda V^{\text{even}}\otimes \Lambda^+ V^{\text{odd}}$ for $i=1, \ldots ,n$ satisfying $dT_{i}=(-1)^{N}x_{i}\omega$. We then  put $f(sx_{i})=T_{i}$. 

For every cocycle $y_{j}$, we determine $S_j$ such that $dS_{j}=(-1)^{N}y_{j}\omega$. We then put $f(sy_{j})=S_{j}$.

Now, for every $y_{j}$ which is not a cocycle i.e.  $dy_j\neq 0$, let $m_j$ denotes the number of $x_i's$ having  degrees less than that of $y_j$. Noticing that
$dy_{j}\in\mathbb{Q}[x_{1},\ldots,x_{n}]$ is  decomposable, we should assume in the following sum:

$$dy_{j}=\sum_{i_{1}\leq \cdots \leq i_{l_j}}\alpha_{i_{1} \cdots i_{l_j}}x_{i_{1}} \cdots x_{i_{l_j}};\; \alpha_{i_{1} \cdots i_{l_j}}\in \mathbb{Q}$$
that $2\leq l_j\leq m_j$. Here, 
by convention, $\alpha_{i_{1} \cdots i_{l_j}}=0$ if $|x_{i_{1}} \cdots x_{i_{l_j}}|\neq |y_j|+1$. Hence,
 
$$sdy_{j}=\sum_{i_{1}\leq \cdots \leq i_{l_j}}\alpha_{i_{1} \cdots i_{l_j}}x_{i_{1}} \cdots x_{i_{l_j-1}}sx_{i_{l_j}}$$  
It follows that,
\begin{eqnarray*}
d(f(sy_{j})) &=& (-1)^{N}f(y_{j})-(-1)^{N}f(sdy_{j})  \\
   &=& (-1)^{N}y_{j}\omega -(-1)^{N}\sum_{i_{1}\leq \cdots \leq i_{l_j}}\alpha_{i_{1} \cdots i_{l_j}}x_{i_{1}}\cdots x_{i_{l_j-1}}f(sx_{i_{l_j}}) \\
   &=& (-1)^{N}y_{j}\omega-(-1)^{N}\sum_{i_{1}\leq \cdots \leq i_{l_j}}\alpha_{i_{1} \cdots i_{l_j}}x_{i_{1}}\cdots x_{i_{l_j-1}}T_{i_{l_j}}.
\end{eqnarray*}

Therefore, knowing  $T_{i_{l_j}}$, the second obstruction is to find $S_{j}\in \Lambda V^{\text{even}}\otimes \Lambda^+ V^{\text{odd}}$ satisfying the differential condition
$$
dS_{j}=(-1)^{N}y_{j}\omega-(-1)^{N}\sum_{i_{1}\leq \cdots \leq i_{l_j}}\alpha_{i_{1} \cdots i_{l_j}}x_{i_{1}}\cdots x_{i_{l_j-1}}T_{i_{l_j}}.
$$
Notice that  $x_{i_{l_j}}$  can reappear from  $T_{i_{l_j}}$.
We then put $f(sy_{j})=S_{j}$.

\item  In the second case, we assume the evaluation map vanishes, i.e. $ev_{\Lambda V}=0$. In this case, referring again to \cite[Corollary 3]{L-M1}, the  cocycle $\omega$, considered in the above case  is called the top class of $(\Lambda V,d)$ and it is actually a coboundary. Let us put $\omega = d\omega'$ with $|\omega'|=N-1$. It is then straightforward to see that the above discussion applies with minor changes:

Indeed, since for every  $1 \leq i \leq n$,  $dT_i = x_i\omega =dx_i\omega'$, we may put $T_i = x_i\omega'$. Therefore, in this case, $x_{i_{l_j}}$  reappears effectively  from  $T_{i_{l_j}}$ and we have:
 $$
 d[f(sy_{j}) - (-1)^{N}y_j\omega'] = [-(-1)^{N}\sum_{i_{1}\leq \cdots \leq i_{l_j}}\alpha_{i_{1}\cdots i_{l_j}} {x}_{i_{l_1}} \cdots x_{i_{l_j}}]\omega'.
 $$
 Thus, we must check for an $S'_j\in \Lambda V^{\text{even}}\otimes \Lambda^+ V^{\text{odd}}$ such that
 $$d(S'_j) = [-(-1)^{N}\sum_{i_{1}\leq \cdots \leq i_{l_j}}\alpha_{i_{1} \cdots i_{l_j}}x_{i_{1}} \cdots x_{i_{l_j}})]\omega'.$$ 
 We then put
 $f(sy_{j}) - (-1)^{N}y_j\omega' = S'_j$ which gives
 $$f(sy_{j}) = (-1)^{N}y_j\omega' + S'_j.$$
 	\end{itemize}
\end{proof}

We now apply this procedure for particular examples.
\begin{example}\label{ExampleA}
Consider the Sullivan model $(\Lambda x_{2},x_{4},y_{5},y_{7}, d)$, $dx_{2}=dx_{4}=0$, $dy_{5}=x_{2}^{3}-2x_{2}x_{4}$, $dy_{7}=x_{4}^{2}-x_{2}^{2}x_{4}$. This is clearly a pure elliptic Sullivan algebra. It is elliptic with formal dimension $N=8$ and the fundamental  class $\omega=[x_{2}^{2}x_{4}-2x_4^2]$.
So, let $f(1)=x_{2}^{2}x_{4}-2x_4^2$. Applying the above algorithm, we get $f(sx_{2})=x_{4}y_{5}$, $f(sx_{4})=x_{2}x_{4}y_{5}+(x_{2}^{2}-2x_4)y_{7}$, $f(sy_{5})=2y_5y_7$ and 
$f(sy_{7})=x_{2}y_{5}y_{7}$. This determines the desired representative $f$ of the generating class $\Omega$ of 
$Ext_{(\Lambda V, d_{\sigma})}(\mathbb{K},(\Lambda V, d_{\sigma}))$.
\end{example}

\begin{example}\label{ExampleB}
Let $(\Lambda x_{2},x_{4},y_{5},y_{7}, d)$, $dx_{2}=dx_{4}=0$, $dy_{5}=-2x_{2}x_{4}$ and $dy_{7}=x_{4}^{2}$. The differential $d=d_2$ is homogeneous of degree two and such model is called a \emph{coformal Sullivan model}. It is the \emph{associated quadratic model}  of $(\Lambda x_{2},x_{4},y_{5},y_{7}, d)$. It is pure but non-elliptic model, since $[x_2]^t\neq0$ for any $t\geq 1$. The class $\omega_0 = [2x_{4}^{2}]$ (whose degree  equals $N=8$), although it is zero, it remains of interest for the determination of a representing cocycle of, say  $\Omega_0$ of $Ext_{(\Lambda V, d)}(\mathbb{Q},(\Lambda V, d))$. $\omega_0$ is called the  {\it top class} of 
$(\Lambda V,d)$ \cite{L-M1}. We put (formally) $f(1)= 2x_{4}^{2}$ (as in the second case $(ii)$ of the proof above). This yields, respectively, :

\begin{itemize}
	\item $d(f(sx_{2}))=f(x_{2})=x_2f(1)=2x_2x_4^2=d_2(ax_4y_5+bx_2y_7)$, with $a$ and $b$ such that $-2a+b=2$, that is $b=2a+2$. Therefore, 
	$f(sx_2)=ax_4y_5+2(a+1)x_2y_7)$, some $a\in\mathbb{Q}$.
	
	\item $d(f(sx_{4}))=f(x_{4})=2x_4^2=d(2x_{4}y_{7})$. Hence,  $f(sx_{4})=2x_{4}y_{7}$.
	
	\item $d(f(sy_{5}))=f(y_5)-f(sd(y_5))=2x_4^2y_5+2x_2f(sx_4)=2x_4^2y_5+ 4x_2x_4y_7=d(-2y_5y_7)$. Hence, $f(sy_{5})=-2y_{5}y_{7}$.
	
	\item $df(sy_{7})=f(y_7)-f(sd(y_7))=2x_4^2y_7-x_4f(sx_4)=2x_4^2y_7- 2x_4^2y_7=0.$
	Hence, $f(sy_{7})=0$. 
\end{itemize}

This determines $\Omega_0=[f]$.
\end{example}

\section{A basis of $Ext_{(\Lambda V,d)}(\mathbb{K},(\Lambda V,d))$ from that of $Ext_{(\Lambda V,d_{k})}(\mathbb{K},(\Lambda V,d_{k}))$}
We consider $(\Lambda V,d_{\Lambda V})$ to be a minimal Sullivan algebra with $d_{\Lambda V}=\sum_{i\geq k}d_{i}$, for some fixed $k\geq 2$,  with $d_{i}(V)\subset \Lambda^{i}V$. We assume that $\dim V < \infty$ and that $(\Lambda V, d)$ is elliptic, but $(\Lambda V, d_k)$ is not necessarily elliptic. As a consequence,
$(\Lambda V,d_{k})$ and $(\Lambda V,  d_{\Lambda V})$ are both Gorenstein differential graded algebras with the same formal dimension given in (\ref{fd}). Recall from \S 2 that they are
 related by  the above spectral sequence \eqref{Ext Milnor Moore spec}
induced by the  filtration (\ref{filt-Ext-MM spec}) which we recall here for convenience:
$$F^{p}A=\{g\in A | \; g(\Lambda sV) \subset \Lambda^{\geq p} V)\}$$
with  $A:= Hom_{\Lambda V}(\Lambda V\otimes \Lambda sV,\Lambda V)$ is equipped with the differential $D_A$ also decomposed as follows
$D =\sum_{i\geq k}D_i$  and satisfying 
$$D_A(f) = d_{\Lambda V}\circ f + (-1)^{|f|+1} f\circ D_{\Lambda V \otimes \Lambda sV}.$$

Henceforth, we denote $d_{\Lambda V}$ by $d$ and both $D_A$ and $D_{\Lambda V \otimes \Lambda sV}$ by $D$. Recall that, due to degree constrains, $d_k$  is a differential on $\Lambda V$ defining $(D_{\Lambda V \otimes \Lambda sV})_k$ 
on $\Lambda V \otimes \Lambda sV$ in the same way that $d_{\Lambda V}$ defined $ D_{\Lambda V \otimes \Lambda sV}$. The induced  $D_k$ on $A$, defined by these differentials, is a differential for degree reasons. 
The first term of \eqref{Ext Milnor Moore spec} is denoted 
$$E_k^{p,q}=H^{p,q}(Hom_{\Lambda V}(\Lambda V\otimes \Lambda sV),\Lambda V), D_{k})$$
 where $p$ is the filtration degree and $q$ its complement in such that $p+q$ stands for the total degree.
 

The Gorenstein structure of $(\Lambda V,d_{k})$ implies the existence of a unique  $(p,q)$ such that $Ext_{(\Lambda V,d_{k})}(\mathbb{K},(\Lambda V,d_{k}))=Ext^{p,q}_{(\Lambda V,d_{k})}(\mathbb{K},(\Lambda V,d_{k}))$ has a single generating class $[f]$ of degree $N=p+q$.
 Using notation of \S 2, we have $f\in Z^{p,q}_{k-1}$ is the unique cocycle, up to a coboundary. Thus,
 \begin{equation}\label{generator f}
f(\Lambda sV)\subset (\Lambda^{\geq p} V)^N\;  \hbox{and}\; Df(\Lambda sV)\subset (\Lambda^{\geq p+k-1} V)^{N+1}
 \end{equation}

\begin{enumerate}
	\item Suppose that $Df=0$.

If moreover $f=Dg$, for some $g\in A^{N-1}$, then $f=D_{k}g+\sum_{i\geq k+1}D_{i}g$. So, $f-D_{k}g$ is a $D_k$ cocycle such that $(f-D_{k}g)(\Lambda sV)\subset \Lambda^{\geq p+1}V$. Necessarily,  $[f-D_{k}g]=0$ in $A$. Otherwise it should be  a  generating class of $Ext_{(\Lambda V,d_{k})}(\mathbb{K},(\Lambda V,d_{k}))$. But, being of  bi-degree $(p+1, q-1)$, we get a contradiction to the uniqueness of $(p,q)$. As a consequence, $f$ is not a coboundary and its class is the desired generator (in this case) of $A$ or equivalently, the spectral sequence \eqref{Ext Milnor Moore spec} degenerates at the first term $E_{k}$.

\item Assume that $Df\neq0$.

Since $Df$ is a $(\Lambda V,d)$-morphism,  it is defined by its restriction to $\Lambda sV$. 
Let, for any $i \geq 0$, $h^{0}_{i}:\Lambda V\otimes \Lambda sV\rightarrow \Lambda V$ ($i \geq 1$), the $(\Lambda V,d)$-module given, on $\Lambda V$, by  $h^{0}_{i}(1)$ (coming from $Df(1)= d(f(1))$),  its restriction  $h_i^0: \Lambda sV\rightarrow \Lambda^{p+k-1+i}V$ and, the  extensions  defined by $h^{0}_{i}(\alpha \otimes sv)=\alpha h_{i}^{0}(sv)$.
For degree reason, there is a fix integer $t\geq 0$ such that
$$Df_{|sV}=h^{0}_{0}+\cdots+h^{0}_{t}.$$
(with the convention that the last term $h^0_t$ is non zero).
Indeed, let  $r= max\{|v_i|,\; v_i\in V\}$ (referring to \cite[Corollary 1, p. 441]{FHTb} $r\leq 2N-1$). As $h_t^0(sv)\in \Lambda^{p+k-1+t}V$, then for all $sv\in sV$,  $|h_t^0(sv)|=|h_t^0|+|sv|=N+1+|sv|\geq 2(p+k-1+t)$ which implies that $N+r\geq 2(p+k-1+t)$. Thus, $t\leq{(N+2-2p-2k+r)}/{2}$.

 This is equivalent to :
\begin{equation}\label{Df(k)}
 Df(\Lambda sV)\subseteq \Lambda^{p+k-1}V\oplus\cdots\oplus\Lambda^{p+k-1+t}V.
 \end{equation}
Until the last of section, we consider only the restrictions to $\Lambda sV$.

Now, since $D^{2}f=Dh^{0}_{0}+\cdots+Dh^{0}_{t}=0$ and, by word-length argument, $D_{k}h^{0}_{0}=0$ (in fact, $D_{k}h^{0}_{0}(\Lambda sV)\subset \Lambda ^{2(k-1)+p}V$ is the unique term  which has this least  word-length). But, $|h^{0}_{0}|=N+1>N$ implies that there exist some $\xi_{1}\in Hom_{\Lambda V}(\Lambda V\otimes \Lambda sV;\Lambda V)$ such that $h^{0}_{0}=D_{k}\xi_{1}$.  Moreover, as $h^{0}_{0}(\Lambda sV)=D_{k}\xi_{1}(\Lambda sV)\subset \Lambda ^{p+k-1}V$ we see that $\xi_{1}(\Lambda sV)\subset \Lambda ^{p}V$ with total degree $N={p+q}$.

We now introduce $f^{1}=f-\xi_{1}$. 

On one hand, Since, $Df^{1} = Df - D\xi_1 = h^{0}_{0}+\cdots+h^{0}_{t} - D_k\xi_1 - (D-D_{k})\xi_1 = (h^{0}_{1}- D_{k+1}\xi_1) + h^{0}_{2} + \cdots + h^{0}_{t} - D_{\geq k+2}\xi_1$ and
$D^2f^1 = 0$, we have, for degree reason, $D_k(h^{0}_{1}- D_{k+1}\xi_1) =0$. So, being $|h^{0}_{1}- D_{k+1}\xi_1| = N+1$, there is some $\xi_2\in A^{N}$ such that $h^{0}_{1}- D_{k+1}\xi_1 = D_k\xi_2$.
This can be resumed (for the same $t$ due again for analogous degree reason) as:
$$Df^{1} = h_1^1 + h^{1}_{2}+\cdots+ h^{1}_{t}$$
where $h_1^1 = h^{0}_{1}- D_{k+1}\xi_1  = D_k\xi_2$.

We continue by introducing  $f^2 =f^1 - \xi_2 = f - \xi_1 - \xi_2$ and repeat the process until we reach 

$$
\left\{
\begin{array}{l}
f^t = f^{t-1} - \xi_t = f -\xi_1 - \cdots - \xi_t\\ Df^t=h_t^t = h_t^0 - D_{k+t}\xi_1 - D_{k+t-1}\xi_2 - \cdots - D_{k+1}\xi_t.
\end{array}
\right.
$$

If $Df^t \neq 0$, we continue the above process to get $Df^t =h_t^t = D_k\xi_{t+1}$ (some $\xi_{t+1}\in A^N$). But, by definition of $t$ we have   $D_{\geq k+1}(\xi_{t+1})=0$. Therefore, $Df^t = D\xi_{t+1}$ so that $D(f^t -\xi_{t+1})=0$. We (should)  then take $f^t -\xi_{t+1}$ instead of $f^t$.

Then, to simplify,  we assume that $Df^t= 0$ and show that $f^t$ is not a coboundary.

Assume that $f^t =Dg$ (some $g\in A^{N-1})$.
On one hand, since all terms of the spectral sequence \eqref{Ext Milnor Moore spec} are one-dimensional (as graded vector spaces) we still have $$f^t(\Lambda sV)\subseteq \Lambda^{p}V\oplus \cdots \oplus \Lambda^{p+t}V.$$ Thus, we should have $g(\Lambda sV)\subseteq \Lambda^{p-k+1}V \oplus \cdots \oplus \Lambda^{p+t-k+1}V$ so that $$g =g^0_0 + g_1^0 + \cdots + g_t^0$$ where $g_i^0(\Lambda (sV)\subseteq \Lambda^{p-k+1+i}V$.
 On the other hand, by (\ref{Df(k)}) we should  have  
$$f(\Lambda sV)\subseteq \Lambda^{p}V\oplus\cdots\oplus\Lambda^{p+t}V.$$
Thus, we may decompose $f$ as $$f =f^0_0 + f_1^0 + \cdots + f_t^0$$ where $f_i^0(\Lambda (sV))\subseteq \Lambda^{p+i}V$.
It follows, using word length argument, that  $f_0^0 - \xi_1 = D_kg^0_0$, Hence, 
$D_kf_0^0 - D_k\xi_1 = D_k^2(g_0^0)=0$ and since $D_kf_0^0=0$ (by word length), we get $h_0^0 =D_k\xi_1= 0$. This contradicts the fact that $f\in F^p(\Lambda V)$. 
As a conclusion,   $[f^{t}]$ is a generating class of $Ext^*_{(\Lambda V,d)}(\mathbb{K},(\Lambda V,d))$.
\end{enumerate}
Thus we have established the following
  \begin{theorem}\label{main 2}
 With the above notation, either the spectral sequence \eqref{Ext Milnor Moore spec} degenerate at $E_k$ and   $[f]$ generates $Ext^*_{(\Lambda V,d)}(\mathbb{K},(\Lambda V,d))$ or else, $[f^t]$  is the generating class of $Ext^*_{(\Lambda V,d)}(\mathbb{K},(\Lambda V,d))$. 
  \end{theorem}
  
 \begin{remark}\label{rem4.2}
From   the  algorithm inducing $f^t$, the following formulas arise:
$$\left\{ \begin{array}{l}
h_0^0 = D_k\xi_1\\
h_1^1 = h_1^0 - D_{k+1}\xi_1= D_k\xi_2\\
h_2^2 = h_2^0 - D_{k+2}\xi_1 - D_{k+1}\xi_2= D_k\xi_3\\
\vdots\\
h_t^t = h_t^0 - D_{k+t}\xi_1 - D_{k+t-1}\xi_2 - \cdots - D_{k+1}\xi_t= D_k\xi_{t+1}.
\end{array}
\right. 
$$
We now describe how this is used explicitly to determine $f^t$ (cf. Example 4.2 below).

From the equation $Df(1) = d(f(1)) = h_0^0(1) + h_1^0(1) + \cdots h_t^0(1)$ with $f(1)$ represents the generating class $\omega_{0}$ of $H(\Lambda V,d_k)$; we deduce the components $h^0_i(1)$ ($0\leq i\leq t$). 
 These are essential in obtaining iteratively the $f^i$'s ($1\leq i\leq t$) as follows: we first determine $\xi_{1}(1)$ from the first equation, evaluated at $1\in \mathbb{Q}$. This determines $\xi_{1}$ on $\Lambda V$. Then, we determine $\xi_{1}$ on  $\Lambda sV$ and therefore,   $f^1 = f - \xi_{1}$ on $\Lambda V \otimes \Lambda sV$. We proceed inductively using the $(i+1)-th$  equation to determine $\xi_{i}$, hence $f^i = f^{i-1}-\xi_i$  on all  $\Lambda V \otimes \Lambda sV$. As it is shown previously, the process terminates while we reach $f^t$.

Alternatively, we can apply Lechuga-Murillo's algorithm \cite[Theorem 5]{L-M02} (or its extension given in \cite{B-R})  to determine the $\xi_{i}(1)$'s but with one step ahead since the first term in the equation is actually $h_0^0(1)$ instead of $h_1^0(1)$.
Consequently, from $f^t = f - \xi_1 - \xi_{2} - \cdots - \xi_t$ we obtain $f^t(1) = f(1) - \xi_1(1) - \xi_{2}(1) - \cdots - \xi_t(1)$ which indeed yields the fundamental  class $[f^t(1)] = ev_{(\Lambda V,d)}([f^t])$ of $(\Lambda V,d)$, that is, the generating class of $H^N(\Lambda V,d)$.
\end{remark}

\subsection{Ext-Ginsburg and $\mathrm{R}_0(\Lambda V,d)$ invariants}
\subsubsection{Ext-Ginsburg invariant}
Recall that M. Ginsburg defined \cite{Gin}, in terms of the following topological Milnor-Moore spectral sequence:
\begin{equation}\label{Top-M-M spec-seq}
E_2^{p,q} = Ext_{H_*(\Omega X, \mathbb{Q})}^{p,q}(\mathbb{Q} , \mathbb{Q}))\implies Ext_{C_*(\Omega X, \mathbb{Q})}^{p+q}(\mathbb{Q} , \mathbb{Q})
 \end{equation} 
a homotopy invariant, denoted $l_0(X)$, to be the largest integer $j$ such that the differential $d_j$ of the $j-th$ term $E_j$ is non-zero. He then showed  that $l_0(X) \leq cat_0(X)$.
 For our purpose, we use the algebraic isomorphic spectral sequence \eqref{Milnor Moore spec} to give the following  equivalent definition of $l_0(X)$:
$$l_0(\Lambda V,d) = sup\{ j\geq 0\; | \; d_j\neq 0 \}$$
or equivalently:
$$l_0(\Lambda V,d) = sup\{m\; | \; E_{m+1}^{*,*}(\Lambda V) = E_{\infty}^{*,*}(\Lambda V) \}.$$
In other words, $l_0(\Lambda V,d)+1$ designates, the {\it stage  where the spectral sequence \eqref{Milnor Moore spec} degenerate}.

Notice that the two definitions agree since,  the  first terms form $E_2$ up to $E_{k-2}$  in \eqref{Milnor Moore spec} are all identical to the term $E_0 = (\Lambda V,0)$ and $E_{k-1}\cong (\Lambda V,d_k)$ so that, if  $E_{\infty}^{p,q}=E_{s+1}^{p,q}$ in (\ref{Top-M-M spec-seq}) and $E_{\infty}^{p,q}(\Lambda V)=E_{t+1}^{p,q}(\Lambda V)$  in \eqref{Milnor Moore spec} (where the first term is $E_{k}^{p,q}(\Lambda V)$), then $t+1$ should be equal to $(s+1-k +1) +(k-1) =s+1$. 

In the same spirit,  we introduced \cite{A-R} in terms of \eqref{Ext Milnor Moore spec} an \emph{Ext-version} of  $l_0(\Lambda V,d)$ as follows:
$$\mathrm{L}_0(\Lambda V,d) = sup\{j\; | \; \delta_j\neq 0\}.$$
(see (\ref{delta}) below for the definition of $\delta_j$).
Clearly an equivalent definition is :
$$\mathrm{L}_0(\Lambda V,d) = sup\{m\; | \; E_{m+1}^{*,*}(A) = E_{\infty}^{*,*}(A)\}.$$
 Now, referring to \cite[Theorem 1.1]{R1} we know that, if $X$ is a simply connected finite type CW-complex, the spectral sequence   \eqref{Ext Milnor Moore spec} is isomorphic to the Eilenberg-Moore spectral sequence:
 \begin{equation}\label{E-M spec-seq}
E_2^{p,q} = Ext_{H_*(\Omega X, \mathbb{Q})}^{p,q}(\mathbb{K} , H_*(\Omega X, \mathbb{Q}))\implies Ext_{C_*(\Omega X, \mathbb{Q})}^{p+q}(\mathbb{K} , C_*(\Omega X, \mathbb{Q})).
 \end{equation}   Let $d_* =(d_j)_j$ (resp. $\delta_*  =(\delta_j)_j$) denote  the  differentials in  \eqref{E-M spec-seq} (resp. \eqref{Ext Milnor Moore spec}). 
 We then introduce the following invariant:
 $$\mathrm{L}_0(X) = max\{j\; | \; d_j\neq 0\}  = max\{m\; | \; E_{m+1}^{*,*} = E_{\infty}^{*,*}\}$$
 Here once again,
$\mathrm{L}_0(\Lambda V,d) $ (resp. $\mathrm{L}_0(X)$) is the order minus one at which the spectral sequence \eqref{Ext Milnor Moore spec} (resp.\eqref{E-M spec-seq}) degenerate.
We call each of these,  \emph{the Ext- Ginsburg invariant}. 

Next, we present an equivalent algorithm, based on spectral sequence argument, to obtain the  generating class $[f^t]$ of $Ext^{p+q}_{(\Lambda V,d)}(\mathbb{K},(\Lambda V,d))$
 from that of $Ext^{p,q}_{(\Lambda V,d_{k})}(\mathbb{K};(\Lambda V,d_{k}))$ denoted here again by $[f]$.
 
Recall from  \S 2 that $f\in Z^{p,q}_{k-1}$. 
That is, $f(\Lambda sV)\subset \Lambda^{\geq p} V$ and $Df(\Lambda sV)\subset (\Lambda^{\geq p+k-1}V)^{N+1}$.
Clearly, if $D(f)=0$ then, as in the above algorithm, $[f]$ generates $Ext^{p+q}_{(\Lambda V,d)}(\mathbb{K},(\Lambda V,d))$, that is \eqref{Ext Milnor Moore spec} collapses at the first term.  We then assume that $Df\neq0$.
 
The decomposition $Df = h_0^0 + h_1^0 + \cdots + h_t^0$ and  the condition $D^2f=0$ lead us to introduce $f^1 = f - \xi_{1}$ where $\xi_1$ is the solution of $h_0^0 =D_k(\xi_1)$. This implies that $\xi_{1}(sV)\subseteq \Lambda ^{p}V$, hence,
\begin{enumerate}
\item $f^{1}(sV) = (f-\xi_{1})(sV)\subseteq  \Lambda ^{\geq p}V$.
\item $Df^{1}(sv) = D(f-\xi_1)(sV)  = [h^{0}_{1}+ \cdots + h^{0}_{t}-(D-D_{k})\xi_{_{1}}](sV)\subseteq \Lambda^{\geq p+k}V$.
\end{enumerate}
Therefore,  $f^{1}\in Z^{p,q}_{k}.$

To continue, recall from \S 2 that the differential $D$ of $A=Hom_{(\Lambda V,d)}(\mathbb{K},(\Lambda V,d))$ induces on
$E_k^{*,*}$ the differential 
\begin{equation}\label{delta}
\begin{array}{llll}
\delta _k : & E_k^{p,q} & \rightarrow & E_k^{p+k,q-k+1}\\
            & [h]_k  &      \mapsto          & [Dh]_k.
\end{array}
\end{equation} 

Applying  this to $h= f^{1}$, we obtain
$$\delta _k[f^{1}]_k = [Df^{1}]_k = [h^{0}_{1}- D_{k+1}\xi_1]_k.$$
Now, we show that $[f^{1}]_k$ is a cocycle representing a non zero class in $(E_{k+1}^{*,*}, \delta_k)$. 

On one hand, Since, 
\begin{equation}
Df^{1} = h^{0}_{0}+\cdots+h^{0}_{t} - D_k\xi_1 - (D-D_{k})\xi_1 = (h^{0}_{1}- D_{k+1}\xi_1) + h^{0}_{2} + \cdots + h^{0}_{t} - D_{\geq k+2}\xi_1
\end{equation}
 and
$D^2f^1 = 0$ we have, by degree reason, $D_k(h^{0}_{1}- D_{k+1}\xi_1) =0$. So, being $|h^{0}_{1}- D_{k+1}\xi_1| = N+1$, there is some $\xi_2\in A$ such that $h^{0}_{1}- D_{k+1}\xi_1 = D_k\xi_2$. Therefore, using the isomorphism $E_k^{p+k,q-k+1} \cong H^{p+q+1}(A,D_k)$, we deduce that $\delta _k[f^{1}]_k=0$ and $[f^{1}]_k$  represents a cohomology class in $E_{k+1}^{p,q}$.

On the other hand,   if  $[f^1]_k = \delta_k[g]_k= [Dg]_k$, for some $g\in  Z_k^{p-k+1,q+k-2}$, then $[f^1 - Dg]_k =0$ in $E_k^{p,q}$ or equivalently, using \eqref{Ekterm}, the component $pr_{-k+1}\circ (f^1 - Dg)$ in $\Lambda^p V$ of $f^1 - Dg$, which we  denote by $(f^1 - Dg)_{\Lambda^p V}$, decomposes as $(h_1 + Dh_2)_{\Lambda^p V}$ in  $Z^{p+1,q-1}_{k-1} + B_{k-1}^{p,q}$. Clearly, only the first term $D_k$ of $D$ acts and then we have $(f^1 - D_kg)_{\Lambda^p V}= (D_k{h_2})_{\Lambda^p V}$. Hence, applying $D_k$ we get  $D_kf^1_{\Lambda^pV} = D_kf_{\Lambda^pV}  - D_k{\xi_1}_{\Lambda^pV}= - D_k{\xi_1}_{\Lambda^pV}=0$ that is, $h_0^0 =0$. This  contradicts the fact that $f\in Z_{k-1}^{p,q}$, hence $[f^{1}]_k$ is indeed a cocycle representing a non zero class in $(E_{k+1}^{*,*}, \delta_k)$. 

We continue by reconsidering the above expression of $Df^1$ which actually  is as follows:
$$Df^{1} = D_k\xi_2 + h^{0}_{2}+\cdots+ h^{0}_{t} - D_{\geq k+2}\xi_1.$$
 This can be resumed (for the same $t$ due again to degree reason) as:
$$Df^{1} = h_1^1 + h^{1}_{2}+\cdots+ h^{1}_{t}$$
where $h_1^1 = h^{0}_{1}- D_{k+1}\xi_1  = D_k\xi_2$.

We then introduce  $f^2 =f^1 - \xi_2 = f - \xi_1 - \xi_2$ and show, using the same argument, that $[f^{2}]_k$ is  a cocycle representing a non zero class in $E_{k+2}^{p,q}$. Moreover, we may write
$$Df^{2} = h_2^2 + h^{2}_{3}+\cdots+ h^{2}_{t}$$
where $h_2^2 = h^{2}_{0}- D_{k+2}\xi_1  - D_{k+1}\xi_2$.

 We continue  this process until we  reach the stage $s$ where the spectral sequence degenerate. This induces  the  generating class of  $E_{\infty}^{p,q} =E_{s}^{p,q}$ represented by  $f^s = f - \xi_1 - \xi_2 - \cdots - \xi_s$. Thus, using the aforementioned identification $E_{\infty}^{p,q} = Ext_{(\Lambda V,d)}^{p+q}(\mathbb{K},(\Lambda V,d))$, we take $[f^s]$ as the  required generating class of $Ext_{(\Lambda V,d)}^{p+q}(\mathbb{K},(\Lambda V,d))$.

As an application, we present below an immediate proof of \cite[Theorem 2]{A-R} which we recall here for convenience: 
\begin{theorem}\label{L-inv}
Let $X$ be a rationally elliptic space. Then,
the spectral sequence \eqref{E-M spec-seq} degenerate at $t$, hence,
$L_0(X) = l_0(X) = t-1$.
\end{theorem}
\begin{proof}
From the discussion  above, the ascendance in the algorithm (à la Lechuga-Murillo) inducing  $[f^t]$ from $[f]$, follows the ascendance in the spectral sequence \eqref{Ext Milnor Moore spec}. Thus, this latter degenerates exactly at $s=t$ and consequently 
$L_0(X) = t-1$. On the other hand,
$ev_{(\Lambda V,d)}([f^t]) = [f^t(1)]$ is the fundamental class of $(\Lambda V,d)$ and it is also  provided by an algorithm (à la Lechuga-Murillo) \cite[Theorem 2]{B-R} from $[f(1)]$ (cf. Remark \ref{rem4.2} above). It results that  the spectral sequence \eqref{Milnor Moore spec} degenerates also at $t$ and consequently 
    $L_0(X) = l_0(X)=t-1$.

\end{proof}

\subsubsection{The invariant $\mathrm{R}_0(\Lambda V,d)$}
Recall from \cite{R1} that  the second author introduced, in terms of \eqref{Ext Milnor Moore spec} the invariant 
$$\mathrm{R}_0(\Lambda V,d) = sup\{m\; | \; E_{\infty}^{m,*} \neq 0\}.$$
This is given in spirit of Toomer's invariant which in turn is given in terms of \eqref{Milnor Moore spec} by:
 $$e_0(\Lambda V,d) := sup\{m\; | \; E_{\infty}^{m,*} \neq 0\}.$$

An equivalent definition of $\mathrm{R}_0(\Lambda V,d)$ given in \cite[Remark 3.4]{R1}, in terms of the generating class $\Omega$ of $Ext^{*}_{(\Lambda V,d)}(\mathbb{K},(\Lambda V,d))$, is as follows:
$$\mathrm{R}_0(\Lambda V,d) = sup\{m\; | \; \Omega \; \hbox{can be represented by a cocycle in}\; F^m(A)\}.$$

Our first result in this section gives another proof of \cite[Theorem 1]{A-R} which relies on the algorithm defining $f^t$:

\begin{theorem}\label{R-inv}
Let $X$ be a rationally elliptic  CW-complex. Then,  $\mathrm{R}_0(X)=e_0(X)=cat_0(X).$ 
\end{theorem}
\begin{proof}
By hypothesis, $X$ being elliptic, its
Sullivan minimal model $(\Lambda V,d)$    is elliptic hence  $\dim V < \infty$ and $\dim H(\Lambda V,d)<\infty$. In particular $(\Lambda V,d_k)$ is a Gorenstein algebras, that is 
$$Ext^{*,*}_{(\Lambda V,d_k)}(\mathbb{K},(\Lambda V,d_k)) = Ext^{p',q'}_{(\Lambda V,d_k)}(\mathbb{K},(\Lambda V,d_k))$$
 where the unique bi-degree $(p',q')$ derived from the filtration (\ref{filt-Ext-MM spec}) satisfies $p'+q'=N$, the formal dimension of $(\Lambda V,d)$. By the convergent spectral sequence \eqref{Ext Milnor Moore spec}, we have also
 $$Ext^{*,*}_{(\Lambda V,d)}(\mathbb{K},(\Lambda V,d)) = Ext^{p',q'}_{(\Lambda V,d)}(\mathbb{K},(\Lambda V,d)).$$
This implies $\mathrm{R}_0(\Lambda V,d) = p' = \mathrm{R}_0(X)$ (cf. \S 6 for more explanation). 
   
Next,  since the chain evaluation map (\ref{cevmap}):
$$\begin{array}{cccc}
cev: & A=Hom_{(\Lambda V,d)}((\Lambda V \otimes \Lambda sV, D) , (\Lambda V,d)) & \rightarrow  & (\Lambda V,d)\\
     &       f          & \mapsto &   f(1)
\end{array}
$$ 
  preserves the filtrations (\ref{filt-Ext-MM spec}) and (\ref{fil-MM-spec}), in homology, it induces a morphism between the convergent  spectral sequences \eqref{Ext Milnor Moore spec} and  \eqref{Milnor Moore spec}. Hence, the the $E_\infty$ morphism $$E_{\infty}(ev): E_{\infty}^{p',q'}(A)\cong Ext^{p',q'}_{(\Lambda V,d)}(\mathbb{K},(\Lambda V,d)) \rightarrow E_{\infty}^{p',q'}(\Lambda V)\cong H^{p'+q'}(\Lambda V,d)$$
    is indeed an isomorphism (recall that   $\dim H^{N}(\Lambda V,d)=1$). By the above theorem, $L_0(\Lambda V,d) = l_0(\Lambda V,d)$ therefore, by (\ref{e_0}) we obtain 
      $e_0{(\Lambda V,d)} = p'=e_0(X)$ and thus, $\mathrm{R}_0(X)=e_0(X)=cat_0(X).$
 
\end{proof}

To illustrate, we apply the algorithm to a specific Sullivan model

  \begin{example}\label{exp4.5}
 Let us  reconsider the Sullivan model studied in example \ref{ExampleA} : $(\Lambda x_{2},x_{4},y_{5},y_{7}, d)$, $dx_{2}=dx_{4}=0$, $dy_{5}=x_{2}^{3}-2x_{2}x_{4}$, $dy_{7}=x_{4}^{2}-x_{2}^{2}x_{4}$. Its formal dimension is $N=8$, the top class is $\omega=[x_{2}^{2}x_{4}]$ and its quadratic associated model is $(\Lambda V,d_{2})=\,(\Lambda x_{2},x_{4},y_{5},y_{7},,\, d)$ with $d_{2}x_{2}=d_{2}x_{4}=0$, $d_{2}y_{5}=-2x_{2}x_{4}$, $d_{2}y_{7}=x_{4}^{2}$. Let $V = Span(x_2, x_4, y_5, y_7)$, thus $(\Lambda V,d_2)$ is pure and non-elliptic and $H^{N}(\Lambda V,d_{2})=\mathbb{Q}[x_{2}^{4}]$.
  
 Recall that,  by using the algorithm described in \S 3, we have also  determined in  Example \ref{ExampleA} (resp. Example \ref{ExampleB}) the generating class $\Omega$ (resp. $\Omega_{2}$) of $Ext^{8}_{(\Lambda V,d)}(\mathbb{K};(\Lambda V,d))$ (resp. $Ext^{(2,6)}_{(\Lambda V,d_{2})}(\mathbb{K};(\Lambda V,d_{2}))$). Now, consider the spectral sequence:
 \begin{center}
 	$Ext^{(2,6)}_{(\Lambda V,d_{2})}(\mathbb{K};(\Lambda V,d_{2}))\implies Ext^{8}_{(\Lambda V,d)}(\mathbb{K};(\Lambda V,d))$.
 \end{center}
 
 Now we determine $\Omega$ by using the Lechuga-Murillo algorithm. Let denote, as in Example \ref{ExampleA}, $\Omega_2 =[f]$ and  consider (formally) $f(1)=2x_{4}^2$ (we can not put  $f(1)=x_2^4$  since then $ev_{(\Lambda V,d_2)}\not=0$ which contradicts non-ellipticity of $(\Lambda V,d_2)$).

The first step will determine $\xi_{1}$. We  begin by its value on $\Lambda V$. This is given by $\xi_{1}(1)$. Using \cite{B-R}  we can use Remark \ref{rem4.2} above or else, Lechuga-Murillo's algorithm as in  \cite[Theorem 5]{L-M02} even if $(\Lambda V,d_2)$ is not elliptic. In particular,  we have $\xi_{1}(1)\in \Lambda ^3V$.
Moreover, 	$|\xi_1(1)|= |f(1)|=8$ which imposes $\xi_1(1)=\mu x_2^2x_4$ (some $\mu \in \mathbb{K}$), hence $d_2(\xi_1(1))=0$.  Thus, 
$$\xi_1(x_2)=\mu x_2^3x_4,\;  \xi_1(x_4)=\mu x_2^2x_4^2,\;  \xi_1(y_5)=\mu x_2^2x_4y_5,\; \hbox{and}\;  \xi_1(y_7)=\mu x_2^2x_4y_7.$$

  Next we consider $\xi_1$ on $\Lambda (sV)$. Recall that $f$ is given on  $\Lambda (sV)$ by Example \ref{ExampleB} with $(Df)(sV)\in\Lambda^{\geq3}V.$ We recall it for our use:

  $$f(sx_2)=ax_4y_5+2(a+1)x_2y_7\;(a\in\mathbb{Q}),\quad f(sx_{4})=2x_{4}y_{7},\quad f(sy_{5})=-2y_{5}y_{7}\;\hbox{and}\; f(sy_{7})=0.$$

\begin{itemize}      
\item  For $sx_2$:
\begin{align*}
	(Df)(sx_{2}) &= d(f(sx_{2}))- f(dsx_{2})\\
	&=d(ax_4y_5 + 2(a+1)x_2y_7)- f(x_{2}) \\
	&= ax_4(x_2^3-2x_2x_4) +2(a+1)x_2(x_4^2-x_2^2x_4) -2x_2x_4^2\\
	& = -(a+2)x_2^3x_4.
\end{align*}

Hence, by word length reason, $h_{0}^{0}(sx_{2})= 0$. On the other side,
\begin{align*}
(D_{2}\xi_{1})(sx_{2})&=(d_2\circ \xi_1 - \xi_1 \circ d_2)(sx_2)\\
&=d_2(\xi_1(sx_2))-\xi_1(x_2)=0.
\end{align*}
Therefore, $d_2(\xi_1(sx_2))=\xi_1(x_2)=\mu x_2^3x_4\in\Lambda^{4}V$. Now, since  $d_2(\xi_1(sx_2))\in\Lambda^{3}V$, necessarily $\mu=0$, hence, on one hand, $\xi_1(1)=0$ and consequently, $\xi_1 =0$ on all $\Lambda V$. On the other hand,  $d_2(\xi_1(sx_2))=0$. 

Now, since, $\xi_1(sx_2)\in (\Lambda^2 V)^9$, we have $\xi_1(sx_2) = a'x_4y_5 + b'x_2y_7$ with $-2a' + b'=0$. It follows that
$\xi_1(sx_2) = a'(x_4y_5 + 2x_2y_7)$, where $a'\in \mathbb{Q}$.

For $f^1 = f -\xi_1$, we have :
\begin{align*}
	f^{1}(x_{2})&=f(x_{2})=2x_2x_4^2 \\
	f^{1}(sx_{2})&=f(sx_{2})- \xi_1(sx_2) \\
	&=  ax_4y_5 + 2(a+1)x_2y_7 - a'(x_4y_5 + 2x_2y_7) \\
	& = \alpha x_4y_5 + 2(\alpha +1) x_2y_7\; (\hbox{some}\;  \alpha = a-a' \in \mathbb{Q}).
\end{align*}
 
\item For $sx_4$:
  \begin{align*}
  (Df)(sx_{4})&=d(f(sx_{4}))-f(dsx_{4})\\
  &=d(f(sx_{4}))-f(x_{4}) \\
  &= d(2x_4y_7)-2x_4^3\\
  &= 2x_{4}^{3}-2x_{2}^{2}x_{4}^{2}-2x_4^3\\
  &= -2x_{2}^{2}x_{4}^{2}
  \end{align*}

Then, as before, $h_{0}^{0}(sx_{4})=D_2(\xi_1)(sx_{4})=0$. On the other side,
\begin{align*}
(D_{2}\xi_{1})(sx_{4})&=(d_2\circ \xi_1-\xi_1\circ d_2)(sx_4)\\
&= d_2(\xi_1(sx_4))-\xi_1(d_2(sx_4))\\
&=d_2(\xi_1(sx_4))-\xi_1(x_4).
\end{align*}
Then $d_2(\xi_1(sx_4))=\xi_1(x_4)= 0$. But, since
 $\xi_1(sx_4)\in (\Lambda^2 V)^{11}$, we have $\xi_1(sx_4)=cx_4y_7$ which leads to $c=0$ and consequently, $\xi_1(sx_4)=0$.

For $f^1 = f -\xi_1$, we get
\begin{align*}
f^{1}(x_{4})&=f(x_{4})=2x_4^3\\f^{1}(sx_{4})&=f(sx_{4})-\xi_{1}(sx_{4})=2x_4y_7.
\end{align*}

\item For $sy_5$:
\begin{align*}
(Df)(sy_{5})&=d(f(sy_{5}))-f(d(sy_{5}))\\
&=d(-2y_5y_7)-f(y_5- sd(y_5))\\ &=-2(x_{2}^{3}-2x_{2}x_{4})y_7+2(x_{4}^{2}-x_{2}^{2}x_{4})y_5- 2x_4^2y_5+f(s(x_2^3-2x_2x_4))\\
&=2ax_2^3y_7 +(a-2)x_2^2x_4y_5.
\end{align*}
 
As above, we have  $h_{0}^{0}(sy_{5})=0.$
Once again,  
\begin{align*}
h_0^0(sy_5)&=(D_{2}\xi_{1})(sy_{5})=(d_2\circ \xi_1-\xi_1\circ d_2)(sy_{5})\\
&=d_2(\xi_1(sy_5))-\xi_1(d_2(sy_5))\\
&= d_2(\xi_1(sy_5))-\xi_1(y_5)+\xi_1(s(-2x_2x_4))\\
&= d_2(\xi_1(sy_5))-\xi_1(y_5)-2x_2\xi_1(sx_4)\\
&= d_2(\xi_1(sy_5)).
\end{align*}
Then $d_2(\xi_1(sy_5))=0$. As before, since
 $\xi_1(sy_5)\in (\Lambda^2 V)^{12}$, we have $\xi_1(sy_5)=ey_5y_7$ which leads to $e=0$ and consequently, $\xi_1(sy_5)=0$.

Therefore, for $f^1$:
\begin{align*}
f^1(y_5)&=f(y_5)=2x_4^2y_5\\ 
f^{1}(sy_{5})&=f(sy_5)=-2y_5y_7.
\end{align*}

\item For $sy_7$:
As $df(sy_{7})=0$, we have 
\begin{align*}
(Df)(sy_{7})&=d(f(sy_{7}))-f(dsy_7)\\
&=-f(y_7- sd(y_7))=-f(y_7)+f(sd(y_7))\\
&=-2x_4^2y_7+f(s(x_4^2-x_2^2x_4))\\
&=-2x_4^2y_7+(x_4-x_2^2)f(sx_4)\\
&=-2x_4^2y_7+(x_4-x_2^2)(2x_4y_7)\\
&=-2x_2^2x_4y_7.
\end{align*}

Once again, we obtain
$h_{0}^{0}(sy_{7})=0$. But, 
\begin{align*}
h_{0}^{0}(sy_{7})&=(D_{2}\xi_{1})(sy_{7})= (d_2\circ \xi_1 - \xi_1 \circ d_2)(sy_{7})\\
&=d_2(\xi_1(sy_7))-\xi_1(d_2(sy_7))\\
&=d_2(\xi_1(sy_7))-\xi_1(y_7)+\xi_1(s(d_2y_7)))\\
&=d_2(\xi_1(sy_7))+\xi_1(sx_4^2)\\
&=d_2(\xi_1(sy_7))+x_4\xi_1(sx_4)\\
&=d_2(\xi_1(sy_7)).
\end{align*}

So, $d_2(\xi_1(sy_7))=0$ and consequently, since $\xi_1(sy_7)\in (\Lambda^2 V)^{14}$, we obtain $\xi_{1}(sy_{7})=0$. 

For $f^1$:
\begin{align*}
f^1(y_7)&=f(y_7)-\xi_1(y_7)=2x_4^2y_7\\ f^{1}(sy_{7})&=f(sy_{7})-\xi_1(sy_7)=0.
\end{align*}

\end{itemize}

We must consider $\xi_2$ because, in the expression for $Df|_{sV}$, there are terms in $\Lambda^4 V$.

We determine firstly $\xi_{2}(1)$ which (for the same reason as for $\xi_{1}(1)$) satisfies: \\
	$\left\{
	\begin{array}{l}
		d_2(\xi_2(1))=0 \\
		\xi_2(1)\in\Lambda^4V,\\
		|\xi_2(1)|=8.
	\end{array}
	\right.$
	Then $\xi_2(1)=\beta x_2^4$,\\  so, 
	$\xi_2(x_2)=\beta  x_2^5$, $\xi_2(x_4)=\beta x_2^4x_4$, $\xi_2(y_5)=\beta  x_2^4y_5$ and  $\xi_2(y_7)=\beta  x_2^4y_7$.
	
Next we consider $\xi_2$ on $\Lambda (sV)$. Recall that $f^1$ is given by 
\begin{align*}
	f^{1}(x_{2})&=f(x_{2})=2x_2x_4^2,\\
	f^{1}(x_{4})&=f(x_{4})=2x_4^3,\\
	f^1(y_5)&=f(y_5)=2x_4^2y_5,\\
	f^1(y_7)&=f(y_7)=2x_4^2y_7,\\
	f^{1}(sx_{2})&=f(sx_{2})-\xi_{1}(sx_2)= \alpha x_4y_5 + 2(\alpha +1) x_2y_7\; (\hbox{some}\;  \alpha = a-a' \in \mathbb{Q}),\\
	f^{1}(sx_{4})&=f(sx_{4})=2x_4y_7,\\
	f^{1}(sy_{5})&=f(sy_5)=-2y_5y_7,\\
	f^{1}(sy_{7})&=f(sy_{7})=0.
\end{align*}

\begin{itemize}
\item  We have $(Df^1)(sx_{2})\in\Lambda^{\geq4}V$ (since $p=2$ and $k=2$), and 
\begin{align*}
	(Df^1)(sx_{2})&=d(f^1(sx_{2}))-f^1(dsx_{2})\\
	&=d(\alpha x_4y_5+2(\alpha +1)x_2y_7)-2x_2x_4^2\\
	&=\alpha x_4(x_{2}^{3}-2x_{2}x_{4})+2(\alpha +1) x_2(x_4^2-x_{2}^2x_{4})-2x_2x_4^2\\
	&=-(\alpha+2)x_2^3x_4.
\end{align*}

 Then  $h_{1}^{1}(sx_{2})= -(\alpha +2) x_2^3x_4$.

\begin{align*}
	(D_{2}\xi_{2})(sx_{2})&=(d_2\circ \xi_2 - \xi_2 \circ d_2)(sx_2)\\
	&=d_2(\xi_2(sx_2))-\xi_2(x_2)=-(\alpha +2) x_2^3x_4.
\end{align*}
Therefore, $d_2(\xi_2(sx_2))=\xi_2(x_2)-(\alpha +2) x_2^3x_4=\beta x_2^5  -(\alpha +2) x_2^3x_4$. 
But, since (i) $x_2^5$ is not a $d_2$-coboundary, (ii)  $d_2(\xi_2(sx_2))\in\Lambda^{4}V$ and (iii) $-(\alpha +2) x_2^3x_4=d_2(\frac{\alpha +2}{2}x_2^2y_5)$, we should have $\beta=0$ and  $\xi_2(sx_2)=\frac{\alpha +2}{2}x_2^2y_5$. In particular, $\xi_2(1) =0$ and therefore, $\xi_2$ is zero on all $\Lambda V$. 

Now, for $f^2 = f^1 -\xi_2 = f- \xi_{1} -\xi_{2}$ as in the algorithm, we deduce from above that
\begin{align*}
	f^{2}(x_{2})&=f(x_{2})=2x_2x_4^2\\
	f^{2}(sx_{2})&=f^1(sx_{2})-\xi_{2}(sx_{2})= (\alpha x_4 - \frac{\alpha +2}{2}x_2^2)y_5 + 2(\alpha +1) x_2y_7.
\end{align*}
\item
\begin{align*}
	(Df^1)(sx_{4})&=d(f^1(sx_{4}))-f^1(dsx_{4})\\
	&=d(f^1(sx_{4}))-f^1(x_{4}) \\
	&=d(2x_4y_7)-2x_4^3\\
	&=2x_{4}^{3}-2x_{2}^{2}x_{4}^{2}-2x_4^3\\
	&=-2x_{2}^{2}x_{4}^{2}.
\end{align*}
Then, $h_{1}^{1}(sx_{4})=D_2(\xi_2)=-2x_{2}^{2}x_{4}^{2}$. On the other side,
\begin{align*}
	(D_{2}\xi_{2})(sx_{4})&=(d_2\circ \xi_2-\xi_2\circ d_2)(sx_4)\\
	&= d_2(\xi_2(sx_4))-\xi_2(d_2(sx_4))\\
	&=d_2(\xi_2(sx_4))-\xi_2(x_4)= d_2(\xi_2(sx_4)).
\end{align*}
Then $d_2(\xi_2(sx_4))=-2x_{2}^{2}x_{4}^{2}= d_2(kx_2x_4y_5 + lx_2^2y_7)$. Thus, $\xi_2(sx_4)=kx_2x_4y_5+lx_2^2y_7$ with $-2k + l = -2$, that is with $l = 2(k-1)$. It follows that $\xi_2(sx_4)= kx_2x_4y_5 + 2(k-1) x_2^2y_7$.

Thus for $f^2 = f^1 -\xi_2$, we get
\begin{align*}
	f^{2}(x_{4})&=f^1(x_{4})-\xi_{2}(x_{4})=2x_4^3\\
	f^{2}(sx_{4})&=f^1(sx_{4})-\xi_{2}(sx_{4})=2x_4y_7-kx_2x_4y_5 - 2(k-1) x_2^2y_7\\
	&=-kx_2x_4y_5+(2x_4-2(k-1)x_2^2)y_7.
\end{align*}

\item \begin{align*}
(Df^1)(sy_{5})&=d(f^1(sy_{5}))-f^1(d(sy_{5}))=d(-2y_5y_7)-f^1(y_5- sd(y_5))\\
&=-2(x_{2}^{3}-2x_{2}x_{4})y_7  + 2(x_{4}^{2}  -x_{2}^{2}x_{4})y_5-2x_4^2y_5+f^1(s(x_2^3-2x_2x_4))\\
&=2\alpha x_2^3y_7 + (\alpha -2)x_2^2x_4y_5.
\end{align*}
 
Then, $h_{1}^{1}(sy_{5})=2\alpha x_2^3y_7 + (\alpha -2)x_2^2x_4y_5.$
Once again,  
\begin{align*}
h_1^1(sy_5)&=(D_{2}\xi_{2})(sy_{5})=(d_2\circ \xi_2-\xi_2\circ d_2)(sy_{5})\\
&=d_2(\xi_2(sy_5))-\xi_2(d_2(sy_5))\\
&= d_2(\xi_2(sy_5))-\xi_2(y_5)+\xi_2(s(-2x_2x_4))\\
&= d_2(\xi_2(sy_5))-2x_2\xi_2(sx_4)\\
&= d_2(\xi_2(sy_5))-2x_2(kx_2x_4y_5 + 2(k-1) x_2^2y_7)\\
&= d_2(\xi_2(sy_5))-2kx_2^2x_4y_5-4(k-1)x_2^3y_7.
\end{align*}
Thus,
\begin{align*}
d_2(\xi_2(sy_5))&= 2kx_2^2x_4y_5+4(k-1)x_2^3y_7+ 2\alpha x_2^3y_7 + (\alpha -2)x_2^2x_4y_5 \\
                &=(2k +\alpha - 2)[x_2^2x_4y_5 + 2x_2^3y_7].
\end{align*}
But $\xi_2(sy_5)\in (\Lambda^3V)^{12}$, hence, it has (uniquely) the form $\xi_2(sy_5)=cx_4^3$ (some $c\in\mathbb{Q}$)  so that $d_2(\xi_2(sy_5))=0$ and consequently, $2k+\alpha-2=0$ i.e. $k=\frac{2-\alpha}{2}$. In particular, actually: 
\begin{align*}
\xi_2(sx_4)= kx_2x_4y_5+2(k-1)x_2^2y_7=\frac{2-\alpha}{2}x_2x_4y_5-\alpha  x_2^2y_7\\
f^{2}(sx_{4})=\frac{\alpha-2}{2}x_2x_4y_5+\alpha x_2^2y_7+2x_4y_7.
\end{align*}
And, 
\begin{align*}
f^2(y_5)&=f^1(y_5)-\xi_2(y_5)=2x_4^2y_5,\\ f^{2}(sy_{5})&=f^1(sy_5)-\xi_2(sy_5)=-2y_5y_7 - cx_4^3 \; (\hbox{some}\;  c\in \mathbb{Q}).
\end{align*}

\item 
\begin{align*}
	(Df^1)(sy_{7})&=d(f^1(sy_{7}))-f^1(dsy_7)\\
	&=-f^1(y_7- sd(y_7))=-2x_4^2y_7+f^1(sd(y_7))\\
	&=-2x_4^2y_7+f^1(s(x_4^2-x_2^2x_4))\\
	&=-2x_4^2y_7+(x_4-x_2^2)f^1(sx_4)\\
	&=-2x_4^2y_7+(x_4-x_2^2)(2x_4y_7)\\
	&=-2x_2^2x_4y_7.
\end{align*}

so that
$h_{1}^{1}(sy_{7})=-2x_2^2x_4y_7$. But, 
\begin{align*}
	h_{1}^{1}(sy_{7})&=(D_{2}\xi_{2})(sy_{7})= (d_2\circ \xi_2 - \xi_2 \circ d_2)(sy_{7})=d_2(\xi_2(sy_7))-\xi_2(d_2(sy_7))\\
	&=d_2(\xi_2(sy_7))-\xi_2(y_7)+\xi_2(s(d_2y_7)))\\
	&=d_2(\xi_2(sy_7))+\xi_2(sx_4^2)\\
	&=d_2(\xi_2(sy_7))+x_4(\frac{2- \alpha}{2}x_2x_4y_5 - \alpha  x_2^2y_7)\\
	&=d_2(\xi_2(sy_7))+ \frac{2-\alpha}{2}x_2x_4^2y_5 - \alpha x_2^2x_4y_7,
\end{align*}
then
\begin{align*}d_2(\xi_2(sy_7)) &=-2x_2^2x_4y_7- \frac{2-\alpha}{2}x_2x_4^2y_5 + \alpha x_2^2x_4y_7\\
                               &= \frac{\alpha -2}{2}x_2x_4^2y_5 + (\alpha -2)x_2^2x_4y_7 = \frac{2 - \alpha}{2}x_2[ - 2x_2x_4y_7 - x_4^2y_5]\\
                               &=  d_2(\frac{2-\alpha}{2}x_2y_5y_7).
 \end{align*}
 It follows that  $\xi_{2}(sy_{7})=\frac{2-\alpha}{2}x_2y_5y_7$. Thus
\begin{align*}
	f^2(y_7)&=f^1(y_7)-\xi_2(y_7)=2x_4^2y_7,\\ f^{2}(sy_{7})& = f^1(sy_{7})-\xi_2(sy_7)= \frac{\alpha -2}{2}x_2y_5y_7.
\end{align*}

\end{itemize}

We next pass to $\xi_3$, hence to $f^3$.
We determine firstly $\xi_{3}(1)$ which (for the same reason as for $\xi_{1}(1)$) satisfies: 

	$\left\{
	\begin{array}{l}
		d_2(\xi_3(1))=0 \\
		\xi_2(1)\in\Lambda^5V,\\
		|\xi_2(1)|=8.
	\end{array}
	\right.$
	
		For degree reason, $\xi_3(1)=0$,  so, 
		$\xi_3$ is zero on all $\Lambda V$ and $f^3 = f^2 = f^1 =f$ on all $\Lambda V$.

	Next we consider $\xi_3$ on $\Lambda (sV)$. Recall that $f^2$ is given by 

	\begin{align*}
		f^{2}(x_{2})&=2x_2x_4^2,\\
		f^{2}(x_{4})&= 2x_4^3,\\
		f^2(y_5)&=2x_4^2y_5,\\
		f^2(y_7)&=2x_4^2y_7,\\ 
		f^2(sx_2)&= (\alpha x_4 - \frac{\alpha +2}{2}x_2^2)y_5 + 2(\alpha +1) x_2y_7,\\
		f^{2}(sx_{4})&=\frac{\alpha-2}{2}x_2x_4y_5+\alpha x_2^2y_7+2x_4y_7,\\
		f^{2}(sy_{5})&=-2y_5y_7 - cx_4^3 \; (\hbox{some}\;  c\in \mathbb{Q}),\\
		f^{2}(sy_{7})& = \frac{\alpha - 2}{2}x_2y_5y_7.
	\end{align*}
	
	\begin{itemize}	
	\item  We have $(Df^2)(sx_{2})\in\Lambda^{\geq4}V$ (since $p=2$ and $k=2$), and 
	\begin{align*}
		(Df^2)(sx_{2}) &= d(f^2(sx_{2}))- f^2(dsx_{2})\\
		&=d\left((\alpha x_4 - \frac{\alpha +2}{2}x_2^2)y_5+2(\alpha+1) x_2y_7\right)-f^2(x_{2})\\
		&=(\alpha x_4-\frac{\alpha +2}{2}x_2^2)(x_2^3-2x_2x_4)+2(\alpha+1) x_2(x_4^2-x_2^2x_4)-2x_2x_4^2\\
		&=-\frac{\alpha+2}{2}x_2^5.
	\end{align*}
	
	Hence, $h_{2}^{2}(sx_{2})=-\frac{\alpha +2}{2}x_2^5$. On the other side,
	\begin{align*}
		(D_{2}\xi_{3})(sx_{2})&=(d_2\circ \xi_3 - \xi_3 \circ d_2)(sx_2)\\
		&=d_2(\xi_3(sx_2))-\xi_3(x_2)=d_2(\xi_3(sx_2))=- \frac{\alpha +2}{2}x_2^5.
	\end{align*}
	But this is not a $d_2$-coboundary, unless if $\alpha = -2$ in which case $d_2(\xi_3(sx_2)) =0$. Now,  $\xi_3(sx_2)\in (\Lambda ^4V)^9$, hence $\xi_3(sx_2) = 0$. 
	
	Now, for $f^3 = f^2 -\xi_3$ as in the algorithm, we deduce from above that (for $\alpha = -2$)
	\begin{align*}
		f^{3}(x_{2})&=f^2(x_{2})=2x_2x_4^2\\
		f^{3}(sx_{2})&=f^2(sx_{2})-\xi_{3}(sx_{2})
		             =-2x_4y_5 - 2x_2y_7.
	\end{align*}

	\item Recall that, for $\alpha=-2$,
	$$f^2(sx_4)=\frac{\alpha-2}{2}x_2x_4y_5+\alpha x_2^2y_7-2x_4y_7= -2x_2x_4y_5-2x_2^2y_7-2x_4y_7=-2x_2x_4y_5-(2x_2^2-2x_4)y_7.$$
	\begin{align*}
		(Df^2)(sx_{4})&=d(f^2(sx_{4}))-f^2(dsx_{4})\\
		&=d(f^2(sx_{4}))-f^2(x_{4})\\
		&=d(-2x_2x_4y_5-(2x_2^2-2x_4)y_7)-2x_4^3\\
		&=-2x_2x_4(x_2^3-2x_2x_4)-(2x_2^2-2x_4)(x_4^2-x_2^2x_4)-2x_4^3\\
		&=0.
	\end{align*}
	Then, $h_{2}^{2}(sx_{4})=D_2(\xi_2)=0$. On the other side,
	\begin{align*}
		(D_{2}\xi_{3})(sx_{4})&=(d_2\circ \xi_3-\xi_3\circ d_2)(sx_4)\\
		&= d_2(\xi_3(sx_4))-\xi_3(d_2(sx_4))\\
		&=d_2(\xi_3(sx_4))-\xi_3(x_4)= d_2(\xi_3(sx_4)).
	\end{align*}
	Then $d_2(\xi_3(sx_4))=0$. Now, $\xi_3(sx_4)\in (\Lambda ^4V)^{11}$. Then, $\xi_3(sx_4)=mx_2^3y_5$ with $d(mx_2^3y_5)=0$ i. e. $m=0$. It follows that $\xi_3(sx_4)=0$
	
	Thus for $f^3 = f^2 -\xi_3$, we get
	\begin{align*}
		f^{3}(x_{4})&=f^2(x_{4})-\xi_{3}(x_{4})=f^2(x_{4})=2x_4^3\\
		f^{3}(sx_{4})&=f^2(sx_{4})-\xi_{3}(sx_{4}) = f^2(sx_{4})\\
		             &=-2x_2x_4y_5  -(2x_2^2 -2x_4)y_7.	            
	\end{align*}
	 
	\item \begin{align*}
		(Df^2)(sy_{5})&=d(f^2(sy_{5}))-f^2(d(sy_{5}))\\
		&=d(-2y_5y_7 - cx_4^3 )-f^2(y_5- sd(y_5))\\ 
		&=-2(x_{2}^{3}-2x_{2}x_{4})y_7+2(x_{4}^{2}-x_{2}^{2}x_{4})y_5- 2x_4^2y_5+f^2(s(x_2^3-2x_2x_4))\\
		&=0.
	\end{align*}
	
	Then, $h_{2}^{2}(sy_{5})=0.$
	Once again,  
	\begin{align*}
		h_2^2(sy_5)&=(D_{2}\xi_{3})(sy_{5})=(d_2\circ \xi_3-\xi_3\circ d_2)(sy_{5})\\
		&=d_2(\xi_3(sy_5))-\xi_3(d_2(sy_5))\\
		&= d_2(\xi_3(sy_5))-\xi_3(y_5)+\xi_3(s(-2x_2x_4))\\
		&= d_2(\xi_3(sy_5))-2x_2\xi_3(sx_4) = d_2(\xi_3(sy_5))\\
	\end{align*}
	then $d_2(\xi_3(sy_5))=0$, so , since $\xi_3(sy_5)\in (\Lambda ^4V)^{12}$, we obtain 	
	$\xi_3(sy_5)=0$.
	
	Thus
	\begin{align*}
		f^3(y_5)&=f^2(y_5)-\xi_3(y_5)=f^2(y_5)=2x_4^2y_5,\\ f^{3}(sy_{5})&=f^2(sy_5)-\xi_3(sy_5)=f^2(sy_5)=-2y_5y_7-cx_4^3.
	\end{align*}

	\item Recall again that, for $\alpha = -2$, $f^2(sx_4) = -2x_2x_4y_5  -(2x_2^2 -2x_4)y_7$ and $f^{2}(sy_{7}) = \frac{\alpha -2}{2}x_2y_5y_7=-2x_2y_5y_7$.
	
	\begin{align*}
		(Df^2)(sy_{7})&=d(f^2(sy_{7}))-f^2(dsy_7)\\
		&=d(f^2(sy_{7}))-f^2(y_7- sd(y_7))=d(-2x_2y_5y_7)-2x_4^2y_7+f^2(sd(y_7))\\
		&=-2x_2(x_2^3-2x_2x_4)y_7 +2x_2(x_4^2-x_2^2 x_4)y_5-2x_4^2y_7+f^2(s(x_4^2-x_2^2x_4))\\
		&=-2x_2(x_2^3-2x_2x_4)y_7 +2x_2(x_4^2-x_2^2 x_4)y_5-2x_4^2y_7+(x_4-x_2^2)f^2(sx_4)\\
		&= 0
	\end{align*}
	
	so that
	$h_{2}^{2}(sy_{7})=0$. Now, 
	\begin{align*}
		h_{2}^{2}(sy_{7})&=(D_{2}\xi_{3})(sy_{7})= (d_2\circ \xi_3 - \xi_3 \circ d_2)(sy_{7})\\
		&=d_2(\xi_3(sy_7))-\xi_3(d_2(sy_7))\\
		&=d_2(\xi_3(sy_7))-\xi_3(y_7)+\xi_3(s(d_2y_7)))\\
		&=d_2(\xi_3(sy_7))+\xi_3(sx_4^2)\\
		&=d_2(\xi_3(sy_7))+x_4\xi_3(sx_4)\\
		&=d_2(\xi_3(sy_7)), 
	\end{align*}
	then $d_2(\xi_3(sy_7))=0$ and consequently, since $\xi_3(sy_7) \in (\Lambda ^4V)^{14}$, $\xi_3(sy_7))=0$.
	It follows that,
	\begin{align*}
		f^3(y_7)&=f^2(y_7)-\xi_3(y_7)=2x_4^2y_7,\\ 
		f^{3}(sy_{7})&=f^{2}(sy_{7})-\xi_{3}(sy_{7})=-2x_2y_5y_7.
	\end{align*} 	
\end{itemize}

\begin{align*}
	f^{3}(x_{2})&=f^2(x_{2})=2x_2x_4^2 \\
	f^{3}(x_{4})&=f^2(x_{4})= 2x_4^3\\
	f^3(y_5)&=f^2(y_5)=2x_4^2y_5,\\
	f^3(y_7)&=f^2(y_7)=2x_4^2y_7,\\ 
	f^{3}(sx_{2})&=f^2(sx_{2})-\xi_{3}(sx_{2})=-2x_4y_5 - 2x_2y_7,\\
	f^{3}(sx_{4})&=f^2(sx_{4})-\xi_{3}(sx_{4})= -2x_2x_4y_5  -(2x_2^2 -2x_4)y_7,\\
	f^{3}(sy_{5})&=f^2(sy_5)-\xi_3(sy_5)=-2y_5y_7-cx_4^3,\\
	f^{3}(sy_{7})&=f^2(sy_{7})-\xi_3(sy_7)=-2x_2y_5y_7.
\end{align*}

Next, we calculate $Df^3$ on $\Lambda sV$.
\begin{itemize}
\item 
\begin{align*} 
Df^3(sx_2) &= d(f^3(sx_2)) - f^3(x_2)+ f^3(s(dx_2)) \\
           &= d(-2x_4y_5 - 2x_2y_7) -2x_2x_4^2\\
           &= -2x_4(x_2^3-2x_2x_4) - 2x_2(x_4^2-x_2^2x_4) -2x_2x_4^2\\
           &= 0
\end{align*}
\item 
\begin{align*} 
Df^3(sx_4) &= d(f^3(sx_4)) - f^3(x_4)+ f^3(s(dx_4)) \\
           &= d(-2x_2x_4y_5  -(2x_2^2 -2x_4)y_7) -2x_4^3\\
           &= -2x_2x_4(x_2^3-2x_2x_4) - (2x_2^2 -2x_4)(x_4^2-x_2^2x_4) -2x_4^3\\
           &= 0
\end{align*}
\item  
\begin{align*} 
Df^3(sy_5) &= d(f^3(sy_5)) - f^3(y_5)+ f^3(s(dy_5)) \\
           &= d(-2y_5y_7-cx_4^3)-2x_4^2y_5+f^3(s(x_2^3-2x_2x_4))\\
           &= -2(x_2^3-2x_2x_4)y_7 +2(x_4^2-x_2^2x_4)y_5 -2x_4^2y_5 + x_2^2f^3(sx_2) -2x_2f^3(sx_4)\\
           &= 0
\end{align*}
\item 
\begin{align*} 
Df^3(sy_7) &= d(f^3(sy_7)) - f^3(y_7)+ f^3(s(dy_7)) \\
           &= d(-2x_2y_5y_7) -2x_4^2y_7 + f^3(s(x_4^2-x_2^2x_4))\\
           &= -2x_2(x_2^3-2x_2x_4)y_7 +2x_2(x_4^2-x_2^2x_4)y_5 -2x_4^2y_7 + (x_4-x_2^2)f^3(sx_4)\\
           &= 0
\end{align*}
\end{itemize}
 
 Finally, we easily check that $f^3$ is not a $D$-coboundary. Thus it is the generating class of $Ext_{(\Lambda V,d)}(\mathbb{Q},(\Lambda V,d))$.
 Hence, $t= 3$ so that $L_0 = t-1 =2$. Moreover, in $(\Lambda V,d)$, we have $[f^t(1)] = [f(1)] = [x_4^2] = [x_2^4]$,  hence, by definition of $e_0(\Lambda V,d)$, we in fact obtain: 
 $$e_0(\Lambda V,d) = e_0([f^3(1)]) = e_0([f(1)]) = e_0([x_2^4]) = 4.$$ 
 
\end{example}

  \section{A basis of $Ext_{(\Lambda V,d)}(\mathbb{K},(\Lambda V,d))$ from that of $Ext_{(\Lambda V,d_{\sigma})}(\mathbb{K};(\Lambda V,d_{\sigma}))$}
  Let $(\Lambda V,d)$ be a minimal Sullivan model with $V$ finite-dimensional. In \cite{FHTb, H77}, denoting $Q= V^{\text{even}}$ and $P = V^{\text{odd}}$, S. Halperin defines a bi-grading $(\Lambda V)^{q+n, -q} = (\Lambda Q \otimes \Lambda^q P)^n$ on $\Lambda V$  with a new bi-grading. Here $n$ is the total degree and $p=n+q$ is the filtering degree. He defines on it the filtration
   $F^{p}(\Lambda V) = (\Lambda V)^{\geq p, *}$
   and shows that it yields a convergent cohomology spectral sequence whose  $E_0$ term is a Sullivan (minimal) model $(\Lambda V,d_{\sigma})$ with the same generating vector space $V$ and a modified differential denoted $d_{\sigma}$ satisfying
  
  \begin{equation}\label{con-pure}
  dv=d_{\sigma}(v)+\phi(v),\; \phi(v) \in \Lambda^{+} V^{\text{odd}}\otimes \Lambda V^{\text{even}}.
  \end{equation}
   More explicitly,
  $d_{\sigma}$ is the part of the differential $d$ satisfying : $d_{\sigma}(V^{\text{even}})=0$ and $d_{\sigma}(V^{\text{odd}})\subseteq \Lambda V^{\text{even}}$. This model is called the \emph{pure model associated to $(\Lambda V,d)$} and the aforementioned spectral sequence:

   \begin{equation}\label{first odd spectral sequence}
   E_2^{p,-q} = H^{p,-q} (\Lambda V,d_{\sigma}) \implies H^{p-q}(\Lambda V,d)
   \end{equation}
   is called the \emph{odd spectral sequence} of $(\Lambda V,d)$.

  The main result, derived from (\ref{first odd spectral sequence}) states that $(\Lambda V,d)$ is elliptic if and only if $(\Lambda V,d_{\sigma})$ is elliptic.

   Thus, as noted in \S 2, a Sullivan model with finite generating space $V$ and differential satisfying the condition \eqref{con-pure} is called a \emph{pure Sullivan (minimal) model}.

  We now consider the acyclic closure $(\Lambda V\otimes \Lambda sV, d)\stackrel{\simeq}{\rightarrow} (\mathbb{Q},0)$. Recall that the differential $D$ on $A=Hom_{\Lambda V}(\Lambda V\otimes\Lambda sV;\Lambda V)$ is given by:
  $$D(f)= f\circ d_{\Lambda V \otimes \Lambda sV} + (-1)^{|f|+1}d_{\Lambda V}\circ f$$
  for any $f\in A$.
  
  Similarly, the differential $d_{\sigma}$ induces  $D_{\sigma}$  on $A$ defined by $$D_{\sigma}(f)= f\circ {d_{\sigma}}_{\Lambda V \otimes \Lambda sV} + (-1)^{|f|+1}{d_{\sigma}}_{\Lambda V}\circ f.$$ Therefore, we obtain the relation  $D(f)=D_{\sigma}(f)+\varphi(f),\; \forall f\in A$ with
$\varphi(f) = (D-D_{\sigma})(f)$.

  In \cite{R1}, denoting
 $$(\Lambda V\otimes\Lambda sV)^{r,s}=\bigoplus_{r=a_{1}+a_2+2a_{3}}(\Lambda Q\otimes \Lambda^{a_{1}}sQ \otimes \Lambda ^{a_2}P\otimes  \Lambda^{a_{3}}sP)^{s},$$
 the second author introduced on
 $A^n=Hom_{\Lambda V}^n(\Lambda V\otimes\Lambda sV,\Lambda V)$ ($n\geq 0$) the following filtration:
\begin{equation}\label{second filtration}
 F^{p}A^{n}=\bigoplus_{r,s}Hom_{\Lambda V}((\Lambda V\otimes\Lambda sV)^{r,s},(F^{p+n+s+r}(\Lambda V))^{n+s}).
\end{equation}
That is, $f\in F^{p}A^{n}$ if and only if $f((\Lambda V\otimes\Lambda sV)^{r,s}) \subseteq (\Lambda ^{\geq p+n+s+r, *}V)^{n+s}$. Note that the differential on $(\Lambda V\otimes\Lambda sV)^{r,s}$ preserves both $s$ (total degree) and $r$ (filtering degree).
Thus, he obtains the following convergent spectral sequence, called the \emph{$Ext$-odd spectral sequence} of $(\Lambda V,d)$.
  \begin{equation}\label{second spectral sequence}
  Ext^{p,-q}_{(\Lambda V,d_{\sigma})}(\mathbb{K},(\Lambda V,d_{\sigma}))\implies Ext^{p-q}_{(\Lambda V,d)}(\mathbb{K},(\Lambda V,d)).
  \end{equation}

Since $V$ is finite-dimensional, both $(\Lambda V,d)$ and $(\Lambda V,d_{\sigma})$ are Gorenstein dga's with the same formal dimension $N$. Consequently, there exists a unique bi-degree $(p,q)$ such that
$Ext_{(\Lambda V,d_{\sigma})}^{*,*}(\mathbb{K},(\Lambda V,d_{\sigma}))= Ext^{(p,-q)}_{(\Lambda V,d_{\sigma})}(\mathbb{K},(\Lambda V,d_{\sigma}))$ is one-dimensional. Recall also that $A:= Hom_{\Lambda V}(\Lambda V\otimes\Lambda sV,\Lambda V)$

Following the approach for the spectral sequence \eqref{Ext Milnor Moore spec}, we proceed to determine a generating class $[h]$ of $Ext^{p-q}_{(\Lambda V,d)}(\mathbb{K},(\Lambda V,d))$ from that, say $[f]$, of $Ext^{p,-q}_{(\Lambda V,d_{\sigma})}(\mathbb{K},(\Lambda V,d_{\sigma}))$.
Using notation of \S 1.4, we have
 $f\in Z_{1}^{p,-q}$ and $|f|=N= p-q$, that is,
 $f(\Lambda V\otimes \Lambda sV)^{r,s}\in  F^{p+N+r+s}(\Lambda V)$ and
 $Df(\Lambda V\otimes \Lambda sV)^{r,s}\subset F^{p+N+r+s+1}(\Lambda V).$
\begin{itemize}
	\item 
First, assume $Df=0$. If, moreover, $f = Dg$, for some $g\in A^{N-1}$, then $f-D_{\sigma}(g) = (D-D_{\sigma})(g)$ is a $D_{\sigma}$-cocycle. Now, as $|g|=N-1$ and $d_{| {\Lambda V \otimes \Lambda sV}}$ sends $(\Lambda V\otimes \Lambda sV)^{r,s}$ to $(\Lambda V\otimes \Lambda sV)^{r-1,s+1}\oplus (\Lambda V\otimes \Lambda sV)^{\geq r,s+1}$, $g$ should send $(\Lambda V\otimes \Lambda sV)^{r-1,s+1}$ to $F^{p+N+r+s}(\Lambda V)=F^{(p+1)+(N-1)+(r-1)+(s+1)}(\Lambda V)$.
 This yields the following diagram :
\begin{equation}\label{diagram1}
  \begin{array}{ccc}
    (\Lambda V\otimes \Lambda sV)^{r,s} & \stackrel{{d_{\sigma}}_{|{\Lambda V \otimes \Lambda sV}}}{\longrightarrow} & (\Lambda V\otimes \Lambda sV)^{r-1,s+1} \\
    \downarrow g &   & g \downarrow \\
    F^{(p+1)+(N-1)+r+s}(\Lambda V )  & \stackrel{{d_{\sigma}}_{|{\Lambda V}}}{\longrightarrow} & F^{p+N+r+s}(\Lambda V)
  \end{array}
\end{equation}
Therefore, $D_{\sigma}(g) = g\circ {d_{\sigma}}_{|{\Lambda V \otimes \Lambda sV}}  + (-1)^N {d_{\sigma}}_{|{\Lambda V}}\circ g$ sends  $(\Lambda V\otimes \Lambda sV)^{r,s}$ to $F^{p+N+r+s}(\Lambda V)$.
The same analysis give us the following diagram:
\begin{equation}\label{diagram2}
  \begin{array}{ccc}
    (\Lambda V\otimes \Lambda sV)^{r,s} & \stackrel{{(d-d_{\sigma})}_{|{\Lambda V \otimes \Lambda sV}}}{\longrightarrow} & (\Lambda V\otimes \Lambda sV)^{\geq r,s+1} \\
    \downarrow g &   & g \downarrow \\
    F^{(p+1)+(N-1)+r+s}(\Lambda V)  & \stackrel{{(d-d_{\sigma})}_{|{\Lambda V}}}{\longrightarrow} & F^{>p+N+r+s}(\Lambda V).
  \end{array}
\end{equation}
It follows that $f-D_{\sigma}(g) = (D-D_{\sigma})(g)$ sends $(\Lambda V\otimes \Lambda sV)^{r,s}$ to $F^{(p+1)+N+r+s}(\Lambda V)$.
So, if $f-D_{\sigma}(g)$ is not a $D_{\sigma}$-coboundary, then $[f-D_{\sigma}(g)]$  becomes a generating class of $ Ext^{*,*}_{(\Lambda V,d_{\sigma})}(\mathbb{K},(\Lambda V,d_{\sigma}))$ of bi-degree $(p+1,q-1)$, which contradicts uniqueness of $(p,q)$. We then conclude that, if $Df=0$ then $[f]$ is a generating class of  $ Ext^{*,*}_{(\Lambda V,d)}(\mathbb{K},(\Lambda V,d))$.

\item Second, assume that $Df\neq 0$. Following the same procedure as in the previous section, we may decompose $Df_{|(\Lambda V\otimes \Lambda sV)^{r,s}}$ as:

\begin{center}
  $Df_{|(\Lambda V\otimes \Lambda sV)^{r,s}} = g_{0}^{0} + \cdots + g_{l}^{0}$
\end{center}
(with, once again, the convention that $l$ is the greatest integer such that $g_{l}^{0}\neq 0$). Here
$l$ is some fixed integer and $g_{i}^{0}$ the $i$-th component of $Df$ sending $(\Lambda V\otimes \Lambda sV)^{r,s}$ to   $(\Lambda V )^{p+N+s+r+i+1,*}$.
Indeed, let  $m= max\{|v_i|,\; v_i\in V\}$ (referring again to \cite[Corollary 1, p. 441]{FHTb} $2\leq m\leq 2N-1$). Thus, since $f$ is a $\Lambda V$-module and,   $g_l^0(sv)\in (\Lambda V)^{p+N+s+r+l+1,*}$ then, for all $sv\in sV$,  $|g_l^0(sv)| = |g_l^0| + |sv| = N +1 + |sv|\Longrightarrow N+m \geq 2( p+N+s+r+l+1)$. Thus, $l\leq \frac{-N+m}{2}-(r+s+p+1)$.  

This decomposition extends to $Df$ so that:
\begin{equation}
  Df=h_{0}^{0}+\cdots+h_{l}^{0}.
\end{equation}
As in the above section, this is equivalent to :
\begin{equation}\label{Df(sigma)}
 Df((\Lambda V\otimes\Lambda sV)^{r,s})\subseteq (\Lambda V)^{p+N+s+r+1,*} \oplus\cdots\oplus (\Lambda V)^{p+N+s+r+l+1,*}.
 \end{equation}
 
Now, using the fact $D^{2}f=0$ and by considering word length, we get  $D_{\sigma}(h_{0}^{0})=0$. Moreover, as $|h_0^0| = |Df| = N+1$, we have $[h_0^0] =0$ hence, $h_{0}^{0}=D_{\sigma}\beta_{1}$ for some $\beta_{1} \in A^N$.

 To continue, we consider  $f^{1}=f-\beta_{1}$. 
 Since, $|\beta_1|=N$, $h_{0}^{0}=D_{\sigma}\beta_{1}$ and  $h_0^0(\Lambda V\otimes \Lambda sV)^{r,s}\subseteq  F^{p+N+r+s+1}(\Lambda V )$ we have $\beta_1(\Lambda V\otimes \Lambda sV)^{r,s}\subseteq  F^{p+N+r+s+1}(\Lambda V )$ and $(D-D_{\sigma})(\beta_1)(\Lambda V\otimes \Lambda sV)^{r,s}\subseteq  F^{p+N+r+s+2}(\Lambda V )$  (compare   with $g$ and $(D-D_{\sigma})(g)$ in the diagram (\ref{diagram1}) and (\ref{diagram2}) respectively). It follows that $f^{1}(\Lambda V\otimes \Lambda sV)^{r,s}\subseteq F^{p+N+r+s+1}(\Lambda V )$. Afterwards, since 
 $$Df_1 =Df - D\beta_1 = (D-D_{\sigma})(f) - h_0^0 - (D-D_{\sigma})(\beta_1) = h_1^0 + \cdots + h_l^0 -(D-D_{\sigma})(\beta_1)$$ we deduce that $Df^{1}(\Lambda V\otimes \Lambda sV)^{r,s}\subseteq F^{p+N+r+s+2}(\Lambda V)$.

 Next, expressing  $F^{p+N+r+s+2}(\Lambda V ) = (\Lambda V) ^{\geq p+N+r+s+2, *}$ as the sum:
 $$F^{p+N+r+s+2} =  (\Lambda V) ^{p+N+r+s+2, *} \oplus (\Lambda V) ^{p+N+r+s+3, *} \oplus \cdots \oplus (\Lambda V) ^{ p+N+r+s+l+1, *}(\Lambda V )$$
 we see that $(D-D_{\sigma})(\beta_1)$ may be decomposed as follows:
 $$(D-D_{\sigma})(\beta_1) = \beta_1^{(1)} + \beta_1^{(2)} + \cdots + \beta_1^{(l)}$$
 where $\beta_1^{(i)}$, ($1\leq i \leq l$), is the component of $(D-D_{\sigma})(\beta_1)$ which sends $(\Lambda V\otimes \Lambda sV)^{r,s}$ to $(\Lambda V)^{p+N+r+s+i+1, *}$.
 Therefore, we may write:
 $$Df^1 = h_1^1  + h_2^1 = \cdots + h_l^1$$
 with $h_1^1 = h_1^0 - \beta_1^{(1)}$.
 We then continue by analyzing the behaviors of $f^1$.

 Again, using the equation $D^2f^1=0$, we obtain some $\beta_2\in A^{N}$ such that $h_1^1 = D_{\sigma}\beta_2$. This leads to $f^2 = f^1-\beta_2 = f -\beta_1 - \beta_2$.  Similarly, we may show that $\beta_2(\Lambda V\otimes \Lambda sV)^{r,s}\subseteq  F^{p+N+r+s+2}(\Lambda V )$ and  $(D-D_{\sigma})(\beta_2)(\Lambda V\otimes \Lambda sV)^{r,s}\subseteq  F^{p+N+r+s+3}(\Lambda V )$, so that  $(D-D_{\sigma})(\beta_2)$ can be expressed as follows:
$$(D-D_{\sigma})(\beta_2) = \beta_2^{(2)} + \beta_2^{(3)} + \cdots + \beta_2^{(l)}.$$
where $\beta_2^{(i)}$, ($2\leq i \leq l$) is the component of $(D-D_{\sigma})(\beta_2)$ which sends $(\Lambda V\otimes \Lambda sV)^{r,s}$ to $(\Lambda V) ^{p+N+r+s+i+1, *}$.

Continuing this process, we inductively define a sequence  $(f^{j}, \beta_{j})\in A^N \times A^N$ for $1\leq j \leq l$ with $f^0 := f$:
$$
\left \{
\begin{array}{l}
f^{j}=f^{j-1}-\beta_{j}= f -\beta_1 -\beta_2 - \cdots - \beta_j\\
Df^{j}(\Lambda V\otimes \Lambda sV)^{r,s}\subseteq F^{p+r+s+N+j+1,*}\\
Df^j =  h_j^j  + h_{j+1}^j = \cdots + h_l^j
\end{array}
\right. .
$$

It remains to show that  $f^{l} = f -\beta_1 -\beta_2 - \cdots - \beta_l$ is indeed a cocycle representing the generating class of $Ext^{p-q}_{(\Lambda V,d)}(\mathbb{K},(\Lambda V,d)).$ 

 If $Df^l \neq 0$, we proceed using the same procedure as in the previous section to obtain $Df^l =h_l^l = D_\sigma\beta_{l+1}$ (for some $\beta_{l+1}\in A^N$). Additionally, based on the definition of $l$, we find that $(D-D_{\sigma})(\beta_{l+1})=0$. Consequently, $Df^l =h_l^l = D\beta_{l+1}$ which implies $D(f^l -\beta_{l+1})=0$. Subsequently, we may substitute $f^l -\beta_{l+1}$ for $f^l$.

Next, we assume for simplicity that $Df^l= 0$ and  show that $f^l$ is not a coboundary. 

Assuming the contrary i.e. there is some $g\in A^{N-1}$ such that $f^l =Dg$.

On one hand,  using (\ref{Df(sigma)}) we see that 
$$f((\Lambda V\otimes\Lambda sV)^{r,s})\subseteq (\Lambda V)^{p+N+s+r,*} \oplus \cdots \oplus (\Lambda V)^{p+N+s+r+l,*}.$$
Therefore, 
$f$ decomposes as $$f =f^0_0 + f_1^0 + \cdots + f_l^0$$
where $f_i^0(\Lambda V\otimes\Lambda sV)^{r,s})\subseteq (\Lambda V)^{p+N+s+r+i,*}$.

On the other  hand,  since all terms of the spectral sequence (\ref{Ext-odd spec}) are
one-dimensional (as graded vector spaces) we still have $$f^l((\Lambda V\otimes\Lambda sV)^{r,s})\subseteq (\Lambda V)^{p+N+s+r,*} \oplus \cdots \oplus (\Lambda V)^{p+N+s+r+l,*}.$$
Thus, we should have
$$g((\Lambda V\otimes\Lambda sV)^{r,s})\subseteq (\Lambda V)^{p+N+s+r-1,*} \oplus \cdots \oplus (\Lambda V)^{p+N+s+r+l-1,*}$$
so $g$ also decomposes as  $$g =g^0_0 + g_1^0 + \cdots + g_l^0$$ 
with $g_i^0((\Lambda V\otimes\Lambda sV)^{r,s})\subseteq (\Lambda V)^{p+N+s+r+i-1,*}$.

It results, due to degree constraints, that 
$f_0^0 - \beta_1 = D_\sigma (g_0^0)$. Hence, applying $D_{\sigma}$ we get
$D_\sigma \beta_1   = 0$. That is $h_0^0 =0$.
This contradicts (\ref{Df(sigma)}) or equivalently the fact that $f\in F^p(A^N)$ and concludes the algorithm.
\end{itemize}

\begin{remark}
A similar  description given in Remark \ref{rem4.2} may be used in terms of the following system of equations to determine $f^l$.
$$
\left \{
\begin{array}{l}
h_0^0 = D_{\sigma}\beta_1\\
h_1^1 = h_1^0 - \beta_1^{(1)}\\
h_2^2 = h_2^0 - \beta_1^{(2)} - \beta_2^{(2)}\\
\vdots\\
h_l^l = h_l^0 - \beta_1^{(l)} - \beta_2^{(l)} - \cdots - \beta_l^{(l)}
\end{array}
\right. 
$$

Alternatively, Lechuga-Murillo’s algorithm \cite{Mu, L-M02} can be applied to determine the $\beta_i(1)$, starting with $h_0^0(1)$ instead of $h_1^0(1)$.

 Here again, from $f^l = f - \beta_1 - \beta_{2} - \cdots - \beta_l$ we obtain $f^l(1) = f(1) - \beta_1(1) - \beta_{2}(1) - \cdots - \beta_l(1)$ which indeed give us the fundamental class $[f^l(1)] = ev_{(\Lambda V,d)}([f^l(1)])$ of $(\Lambda V,d)$ from that of $(\Lambda V,d_{\sigma})$.

\end{remark}

\begin{theorem} With the notation above, 
  $[f^{l}]$ is the generating class of $Ext_{(\Lambda V,d)}(\mathbb{K},(\Lambda V,d))$.
\end{theorem}

We now prove {\bf Theorem \ref{th1.1}} :
\begin{proof} (Theorem \ref{th1.1})
Let $f_{k,\sigma}(1)$  be the  representative of the top class of $(\Lambda V,(d_{\sigma})_{k})$, from which we start our algorithm. This is indeed a  coboundary since
    $[f_{k,\sigma}(1)] = ev_{(\Lambda V,(d_{k})_{\sigma})}([f_{k,\sigma}])=0$ \cite[Corollary 3]{L-M1}. 
    But,  as $(\Lambda V,d_{\sigma})$ is itself elliptic,  $ev_{(\Lambda V,d_{\sigma})}([f_{k,\sigma}^t]) = [f_{k,\sigma}^t(1)]\neq0$, thus,  $(f_{k,\sigma}^t)(1)$ is  not a coboundary.  Clearly, 
    the same holds for $(f_{k,\sigma}^t)^l(1)$ by the convergence of the spectral sequence (\ref{odd spec}). Thus, $[(f_{k,\sigma}^t)^l(1)]$ is the fundamental class of $(\Lambda V,d)$ and
    Consequently, $e_0(\Lambda V,d) = e_0([(f_{k,\sigma}^t)^l(1)])$.
\end{proof}

\begin{example}
	Consider the Sullivan model $(\Lambda(a,b,x,u,v,w), d)$ with degrees $|a|=2$, |b|=4$, |x|=|u|=3$, $|v|=5$ and $|w|=7$. The differentials are defined as $da=dx=0$, $du=a^{2}$, $db=ax$, $dv=ab-ux$, and $dw=b^{2}-2vx$. The associated pure differential is expressed as $d_{\sigma}a=d_{\sigma}b=d_{\sigma}x=0$, $d_{\sigma}u=a^{2}$, $d_{\sigma}v=ab$ and $d_{\sigma}w=b^{2}$. Note that $d = d_{2}$ and $d_{\sigma}=d_{\sigma,2}$. The fundamental class of $(\Lambda V,d_{\sigma})$ is represented by $\omega_{0}=abxv-b^{2}xu$, and the formal dimension $N$ is given by $fd(\Lambda V,d_{ \sigma})=fd(\Lambda V,d)=|\omega_{0}|=14$. The generating class of $Ext_{(\Lambda V,d_{\sigma})}(\mathbb{Q},(\Lambda V,d_{\sigma}))$ is defined by a non-coboundary $\Lambda V$-morphism $f:\Lambda V\otimes(\mathbb{Q}\oplus sV)\rightarrow \Lambda V$ satisfying $f(1)=\omega_{0}=abxv-b^{2}xu$ and $D_{\sigma}(f)=0$ i.e. $d_{\sigma} \circ f = (-1)^{N}f \circ D_{\sigma}$ (here $D_{\sigma}$ is the differential on $\Lambda V\otimes \Lambda sV$).  We now compute $f$. First, we have:
	\begin{align*}
	f(a) = a^2bxw -ab^2xu,\\
	f(b) = ab^2xv - b^3xu,\\
	f(x) = 0, (\hbox{since}\; x^2=0),\\
	f(u) = -abxuv,\\
	f(v) = -b^2xuv,\\
	f(w) = abxvw-b^2xuw.
	\end{align*}
	Next, we determine $f$ on $sV$: 
	\begin{itemize}
		\item 
		\begin{align*}
			d_{\sigma}f(sa) & = (-1)^{N}f(a)-(-1)^{N}f(sd_{\sigma}a)\\
			&=f(a)=af(1) =a^{2}bxv-ab^{2}xu\\
			& =bx(a^{2}v-abu) =bx((d_{\sigma}u)v-ud_{\sigma}v) \\
			& =bxd_{\sigma}(uv) =d_{\sigma}(-bxuv)
		\end{align*}
		so, $f(sa)=-bxuv$.
		\item 
		\begin{align*}
			d_{\sigma}f(sb) &= (-1)^{N}f(b)-(-1)^{N}f(sd_{\sigma}b)\\
			&=f(b)=bf(1)= ab^{2}xv-b^{3}xu \in (\Lambda^{5} V)^{18}
		\end{align*}
		so, we need to find $T\in (\Lambda^{4} V)^{17} $ so that $d_{\sigma }T=d_{\sigma}f(sb)$ and we put $f(sb)=T$. However, $T$ should be of the form $\alpha_{1}axvw+\alpha_{2}bxuw$, thus,
		\begin{align*}
			d_{\sigma}f(sb) & =d_{\sigma}(\alpha_{1}axvw+\alpha_{2}bxuw) =\alpha_{1}d_{\sigma}(ax)vw-\alpha_{1}ax d_{\sigma}(vw)+\alpha_{2}d_{\sigma }(bx)uw-\alpha_{2}bxd_{\sigma}(uw)\\
			& =-\alpha_{1}a^{2}bxw+\alpha_{1}ab^{2}xv-\alpha_{2}a^{2}bxw+\alpha_{2}b^{3}xu \\
			&=-(\alpha_{1}+\alpha_{2})a^{2}bxw+\alpha_{1}ab^{2}xv+\alpha_{2}b^{3}xu
		\end{align*}
		so, $\alpha_{1}+\alpha_{2}=0$, $\alpha_{1}=1$, $\alpha_{2}=-1$. Then, $f(sb)=axvw-bxuw$.
		\item  
		\begin{align*}
			d_{\sigma}f(sx) &= (-1)^{N}f(x)-(-1)^{N}f(sd_{\sigma}x)\\
			& =(-1)^{N}f(x)=0 
		\end{align*}
		so, $f(sx)=0$.
		\item
		\begin{align*}
			d_{\sigma}f(su) &= (-1)^{N}f(u)-(-1)^{N}f(sd_{\sigma}u)\\
			&= f(u)-f(s(a^2)) = f(u)-af(sa) \\
			& =-abxuv+abxuv =0
		\end{align*}
		so, $f(su)=0$.
		\item
		\begin{align*}
			d_{\sigma}f(sv) &= (-1)^{N}f(v)-(-1)^{N}f(sd_{\sigma}v)= f(v)-f(sd_{\sigma}v)\\
			& =-b^{2}xuv-f(s(ab)) = -b^{2}xuv-af(s(b))  \\
			& =-b^{2}xuv-a(axvw-bxuw) =-b^{2}xuv-a^{2}xvw+abxuw \in (\Lambda^{5}V)^{19}\\
			& = d_{\sigma}(xuvw)
		\end{align*}
		so, $f(sv)=xuvw$

		\item
		\begin{align*}
			d_{\sigma}f(sw) &= (-1)^{N}f(w)-(-1)^{N}f(sd_{\sigma}w)= f(w)-f(sd_{\sigma}w)\\
			&=f(w)-f(s(b^2)) =f(w)-bf(sb)   \\
			& = abxvw-b^{2}xuw-b(axvw-bxuw) = abxvw-b^{2}xuw-abxvw+b^{2}xuw = 0
		\end{align*}
		Then, $f(sw)=0$.
	\end{itemize}

	Now we will determine $\beta_{1}$, for this
	\begin{itemize}
		\item \begin{align*}
			Df(1) & =(df-fd)(1)=df(1)=d\omega_{0} \\
			& =d(abxv-b^{2}xu) =d(ab)xv+abd(xv)-d(b^{2})xu-b^{2}d(xu)\\
			& =-adbxv-abxdv-2bdbxu+b^{2}xdu\\
			& =-a^{2}x^{2}v-abx(ab-ux)-2bax^{2}u+a^{2}b^{2}x =-a^{2}b^{2}x+a^{2}b^{2}x =0\\
			& =h_{0}^{0}(1)=D_{\sigma}\beta_{1}(1)=(d\beta_{1}-\beta_{1}d)(1) =d_{\sigma}\beta_{1}(1)=0.
		\end{align*}
Since $\beta_{1}(1)\in (\Lambda^{4} V)^{14}$, we have $\beta_{1}(1)=\alpha_{1}a^{2}xw+\alpha_{2}b^{2}xu+\alpha_{3}abxv+\alpha_{4}a^{2}uw+\alpha_{5}abuv$ for some $\alpha_{1},\alpha_{2},\alpha_{3},\alpha_{4},\alpha_{5}\in \mathbb{Q}$. It follows that:
		\begin{align*}
			d_{\sigma}\beta_{1}(1) & =d_{\sigma}(\alpha_{1}a^{2}xw+\alpha_{2}b^{2}xu+\alpha_{3}abxv+\alpha_{4}a^{2}uw+\alpha_{5}abuv) \\
			& =\alpha_{1}d_{\sigma}(a^{2}xw)+\alpha_{2}d_{\sigma}(b^{2}xu)+\alpha_{3}d_{\sigma}(abxv)+\alpha_{4}d_{\sigma}(a^{2}uw)+\alpha_{5}d_{\sigma}(abuv)\\
			& =-\alpha_{1}a^{2}xb^2-\alpha_{2}b^{2}xa^2-\alpha_{3}abxab+\alpha_{4}a^{2}(a^2w-ub^2)+\alpha_{5}ab(a^2v-uab)\\
			& =-(\alpha_{1}+\alpha_{2}+\alpha_{3})a^{2}b^{2}x-(\alpha_{4}+\alpha_{5})a^2b^2u+\alpha_{4}a^4w+\alpha_{5}a^3bv=0
		\end{align*}
so, $\alpha_{1}+\alpha_{2}+\alpha_{3}=0$, which implies $\alpha_{3}=-\alpha_{1}-\alpha_{2}$, $\alpha_{4}=\alpha_{5}=0$ then
$$\beta_{1}(1)=\alpha_{1}a^{2}xw+\alpha_{2}b^{2}xu-(\alpha_{1}+\alpha_2)abxv
\text{ for some }\alpha_{1},\alpha_{2} \in \mathbb{Q}.$$
We set $\alpha_{1}+\alpha_{2}=0$ to simplify $\beta_1(1)$, then:
$$\beta_{1}(1)=\alpha a^{2}xw-\alpha b^{2}xu
\text{ for some }\alpha_{1} \in \mathbb{Q}.$$
	\end{itemize}
	\begin{itemize}
		\item \begin{align*}
			Df(sa) & =(df-fd)(sa)=df(sa)-f(a-sda)= df(sa)-f(a) \\
			&= d(-bxuv)-a^{2}bxv+ab^{2}xu= -d(bx)uv+bxd(uv)-a^{2}bxv+ ab^{2}xu\\
			&= -ax^{2}uv+a^{2}bxv-bxu(ab-ux)-a^{2}bxv+ab^{2}xu\\
			&= a^{2}bxv-ab^{2}xu-a^{2}bxv+ab^{2}xu=0\\
			&= h_{0}^{0}(sa)=D_{\sigma}\beta_{1}(sa)=(d_{\sigma}\beta_{1}-\beta_{1}d_{\sigma})(sa) = d_{\sigma}\beta_{1}(sa)-\beta_{1}(a)=0
		\end{align*}
		
so
		\begin{align*}
			d_{\sigma}\beta_{1}(sa) & = \beta_{1}(a) =\alpha a^{3}xw-\alpha ab^{2}xu =\alpha ax(a^{2}w-ub^{2})\\
			& =\alpha ax(d_{\sigma}(u)w-ud_{\sigma}w) =\alpha axd_{\sigma}(uw)=d_{\sigma}(-\alpha axuw)
		\end{align*}
		Thus, $\beta_{1}(sa)=-\alpha axuw$, $\alpha \in \mathbb{Q}$.
	\end{itemize}
	\begin{itemize}
		\item \begin{align*}
			Df(sb) & =(df-fd)(sb)=df(sb)-f(b-sdb) =df(sb)-f(b)+af(sx)=df(sb)-f(b) \\
			& =d(axvw-bxuw)-ab^{2}xv+b^{3}xu\\ &=d(ax)vw-axd(vw)-d(bx)uw+bxd(uw)-ab^{2}xv+b^{3}xu\\
			& =-ax(ab-ux)w+axv(b^{2}-2vx)-ax^{2}uw+a^{2}bxw-bxu(b^{2}-2vx)-ab^{2}xv+b^{3}xu\\
			& =-a^2bxw+ab^{2}xv+a^{2}bxw-b^{3}xu-ab^{2}xv+b^{3}xu =0\\
			& =h_{0}^{0}(sb)=D_{\sigma}\beta_{1}(sb)=(d_{\sigma}\beta_{1}-\beta_{1}d_{\sigma})(sb) = d_{\sigma}\beta_{1}(sb)-\beta_{1}(b)
		\end{align*}
so
		\begin{align*}
			d_{\sigma}\beta_{1}(sb) & = \beta_{1}(b)=\alpha a^{2}bxw-\alpha b^{3}xu =\alpha bx(a^{2}w-ub^{2})\\
			& =\alpha bx(d_{\sigma}(u)w-ud_{\sigma}w) =\alpha bxd_{\sigma}(uw)=d_{\sigma}(-\alpha bxuw)
		\end{align*}
		Then $\beta_{1}(sb)=-\alpha bxuw$, $\alpha \in \mathbb{Q}$.
	\end{itemize}
	\begin{itemize}
		\item \begin{align*}
			Df(sx) & =(df-fd)(sx)=df(sx)-f(x-sdx)) =df(sx)-f(x)=-f(x)=0 \\
			& =h_{0}^{0}(sx)=D_{\sigma}\beta_{1}(sx)=(d_{\sigma}\beta_{1}-\beta_{1}d_{\sigma})(sx)\\
			& = d_{\sigma}\beta_{1}(sx)-\beta_{1}(x)+\beta_{1}(sd_{\sigma}x) =d_{\sigma}\beta_{1}(sx)-\beta_{1}(x)
		\end{align*}
		So, $d_{\sigma}\beta_{1}(sx)=\beta_{1}(x)=0$, then $\beta_{1}(sx)=0$.
	\end{itemize}
	\begin{itemize}
		\item \begin{align*}
			Df(su) & =(df-fd)(su)=df(su)-f(u-sdu))  =df(su)-f(u)+af(sa)\\
			&=-f(u)+af(sa) =-abxvu+a(-bxuv)=abxuv-abxuv=0\\
			& =h_{0}^{0}(su)=D_{\sigma}\beta_{1}(su)=(d_{\sigma}\beta_{1}-\beta_{1}d_{\sigma})(su)\\
			& = d_{\sigma}\beta_{1}(su)-\beta_{1}(u)+\beta_{1}(sd_{\sigma}u) =d_{\sigma}\beta_{1}(su)-\beta_{1}(u)+a\beta_{1}(sa)
		\end{align*}
so
		\begin{align*}
			d_{\sigma}\beta_{1}(su) & = \beta_{1}(u)-a\beta_{1}(sa) =\alpha a^{2}xwu-a(-\alpha axuw)\\
			& =-\alpha a^{2}xuw+\alpha a^{2}xuw=0
		\end{align*}
		Then $\beta_{1}(su)=0$.
	\end{itemize}
	\begin{itemize}
		\item
		
		\begin{align*}
			Df(sv) & =(df-fd)(sv)=df(sv)-f(v-sdv) =df(sv)-f(v-s(ab-ux))\\
			& =df(sv)-f(v)+af(sb)-uf(sx) =d(xuvw)+b^{2}xuv+a^{2}xvw-abxuw\\
			& =d(xu)vw+xud(vw)+b^{2}xuv+a^{2}xvw-abxuw\\
			& =-a^{2}xvw+xu(ab-ux)w-xuv(b^{2}-2vx)+b^{2}xuv+a^{2}xvw-abxuw\\
			& =-a^{2}xvw+abxuw-b^{2}xuv+b^{2}xuv+a^{2}xvw-abxuw =0\\
			& =h_{0}^{0}(sv)=D_{\sigma}\beta_{1}(sv)=(d_{\sigma}\beta_{1}-\beta_{1}d_{\sigma})(sv) =d_{\sigma}\beta_{1}(sv)-\beta_{1}(v)+\beta_{1}(sd_{\sigma}v)\\
			& =d_{\sigma}\beta_{1}(sv)-\beta_{1}(v)+\beta_{1}(s(ab)) =d_{\sigma}\beta_{1}(sv)-\beta_{1}(v)+a\beta_{1}(sb)
		\end{align*}
so
		\begin{align*}
			d_{\sigma}\beta_{1}(sv) & = \beta_{1}(v)-a\beta_{1}(sb) =\alpha a^{2}xwv-\alpha b^{2}xuv+\alpha abxuw =\alpha x(a^{2}wv-b^{2}uv+abuw)\\
			& =\alpha x(d_{\sigma}(u)wv-d_{\sigma}(w)uv+d_{\sigma}(v)uw)=\alpha xd_{\sigma}(-uvw)=d_{\sigma}(\alpha xuvw)
		\end{align*}
		Then $\beta_{1}(sv)=\alpha xuvw$, for some $\alpha \in \mathbb{Q}$
	\end{itemize}
	\begin{itemize}
		\item \begin{align*}
			Df(sw) & =(df-fd)(sw)=df(sw)-f(w-sdw) =df(sw)-f(w-s(b^{2}-2vx))\\
			& =df(sw)-f(w)+bf(sb)-2vf(sx)=df(sw)-f(w)+bf(sb) \\
			& =-abxvw+b^{2}xuw+abxvw-b^{2}xuw =0\\
			& =h_{0}^{0}(sw)=D_{\sigma}\beta_{1}(sw)=(d_{\sigma}\beta_{1}-\beta_{1}d_{\sigma})(sw) =d_{\sigma}\beta_{1}(sw)-\beta_{1}(w)+\beta_{1}(sd_{\sigma}w)\\
			& =d_{\sigma}\beta_{1}(sw)-\beta_{1}(w)+\beta_{1}(s(b^{2})) =d_{\sigma}\beta_{1}(sw)-\beta_{1}(w)+b\beta_{1}(sb)
		\end{align*}
so
	\begin{align*}
	d_{\sigma}\beta_{1}(sw) & = \beta_{1}(w)-b\beta_{1}(sb) =-\alpha b^{2}xuw+\alpha b^{2}xuw=0
	\end{align*}
Then $\beta_{1}(sw)=0$.
	\end{itemize}
We now move on to determine $f^{1}$, notice that
	\begin{itemize}

\item	\begin{align*}
f^{1}(1) & =f(1)-\beta_1(1) =abxv-b^{2}xu-\alpha a^{2}xw+\alpha b^{2}xu.\\
f^{1}(a) & =f(a)-\beta_{1}(a) =a^{2}bxv-ab^{2}xu-\alpha a^{3}xw+\alpha ab^{2}xu, \; \hbox{for some}\;  \alpha \in \mathbb{Q}.\\
f^{1}(sa) & =f(sa)-\beta_{1}(sa) =-bxuv+\alpha axuw, \; \hbox{for some}\; \alpha \in \mathbb{Q}
	\end{align*}
We calculate $Df^{1}(sa) $:
	 \begin{align*}
			Df^{1}(sa) & = (df^{1}-f^{1}d)(sa)=df^{1}(sa)-f^{1}(a-sda) =df^{1}(sa)-f^{1}(a) \\
			& =d(-bxuv+\alpha axuw)-a^{2}bxv+ab^{2}xu+\alpha a^{3}xw-\alpha ab^{2}xu \\
			& =-d(bx)uv+bxd(uv)+\alpha d(ax)uw-\alpha axd(uw)-a^{2}bxv+ab^{2}xu+\alpha a^{3}xw-\alpha ab^{2}xu  \\
			&=a^{2}bxv-ab^{2}xu-\alpha a^{3}xw+\alpha ab^{2}xu-a^{2}bxv+ab^{2}xu+\alpha a^{3}xw-\alpha ab^{2}xu =0
		\end{align*}

\item	$f^{1}(b)=f(b)-\beta_{1}(b)=ab^{2}xv-b^{3}xu-\alpha a^{2}bxw+\alpha b^{3}xu$, for some $\alpha \in \mathbb{Q}$.\\
	$f^{1}(sb)=f(sb)-\beta_{1}(sb)=axvw-bxuw+\alpha bxuw$, for some $\alpha \in \mathbb{Q}$.
	
	We calculate $Df^{1}(sb)$:
 \begin{align*}
			Df^{1}(sb) & = (df^{1}-f^{1}d)(sb)=df^{1}(sb)-f^{1}(b-sdb)\\
			& =df^{1}(sb)-f^{1}(b)+f(s(ax))=df^{1}(sb)-f^{1}(b)+af(sx)) =df^{1}(sb)-f^{1}(b) \\
			& =d(axvw-bxuw+\alpha bxuw)- ab^{2}xv+b^{3}xu+\alpha a^{2}bxw-\alpha b^{3}xu\\
			& =-axd(vw)+bxd(uw)-\alpha bxd(uw)- ab^{2}xv+b^{3}xu+\alpha a^{2}bxw-\alpha b^{3}xu\\
			& =-ax(ab-ux)w+axv(b^{2}-2vx)+a^{2}bxw-bxu(b^{2}-2vx)-\alpha a^{2}bxw\\
			&+\alpha bxu(b^{2}-2vx)- ab^{2}xv+b^{3}xu+\alpha a^{2}bxw-\alpha b^{3}xu\\
			& =-a^{2}bxw+ab^{2}xv+a^{2}bxw-b^{3}xu-\alpha a^{2}bxw+\alpha b^{3}xu- ab^{2}xv+b^{3}xu+\alpha a^{2}bxw\\
			&-\alpha b^{3}xu =0
		\end{align*}

\item	$f^{1}(u)=f(u)-\beta_{1}(u)=-abxuv-\alpha a^{2}xwu=-abxuv+\alpha a^{2}xuw$, for some $\alpha \in \mathbb{Q}$.\\
	$f^{1}(su)=f(su)-\beta_{1}(su)=f(su)=0$.
	
		We calculate $Df^{1}(su)$:

	\begin{align*}
			Df^{1}(su) & = (df^{1}-f^{1}d)(su)=df^{1}(su)-f^{1}(u-sdu) =df^{1}(su)-f^{1}(u)+af^{1}(sa)\\
			&  =-f^{1}(u)+af^{1}(sa)  =abxuv-\alpha a^{2}xuw+a(-bxuv+\alpha axuw)\\
			& =abxuv-\alpha a^{2}xuw-abxuv+\alpha a^{2}xuw =0
		\end{align*}

\item	$f^{1}(x)=f(x)-\beta_{1}(x)=0$ and $f^{1}(sx)=f(sx)-\beta_{1}(sx)=0$.

We calculate $Df^{1}(sx)$:
 \begin{align*}
			Df^{1}(sx) & =df^{1}(sx)-f^{1}(x)+af^{1}(sx) \\
			& =df^{1}(sx)-f^{1}(x) =0
		\end{align*}
Clearly, $Df^{1}(sx) = 0$:

\item	$f^{1}(v)=f(v)-\beta_{1}(v)=-b^{2}xuv+\alpha a^{2}xvw+\alpha b^{2}xuv$, for some $\alpha \in \mathbb{Q}$.\\
	$f^{1}(sv)=f(sv)-\beta_{1}(sv)=xuvw-\alpha xuvw$, for some $\alpha \in \mathbb{Q}$.
	
We calculate $Df^{1}(sv)$:
 \begin{align*}
			Df^{1}(sv) & =df^{1}(sv)-f^{1}(v)+f^{1}(sdv) =df^{1}(sv)-f^{1}(v)+f^{1}(s(ab-ux)) \\
			& =df^{1}(sv)-f^{1}(v)+af^{1}(sb)-uf^{1}(sx)= df^{1}(sv)-f^{1}(v)+af^{1}(sb)\\
			& =d(xuvw-\alpha xuvw)+b^{2}xuv-\alpha a^{2}xvw-\alpha b^{2}xuv+a(axvw-bxuw+\alpha bxuw) \\
			& =d(xu)vw+xud(vw)-\alpha d(xu)vw-\alpha xud(vw)+b^{2}xuv-\alpha a^{2}xvw-\alpha b^{2}xuv\\
			&+a^{2}xvw-abxuw+\alpha abxuw\\
			& =-a^{2}xvw+xu(ab-ux)w-xuv(b^{2}-2vx)+\alpha a^{2}xvw-\alpha xu(ab-ux)w\\
			&+\alpha xuv(b^{2}-2vx)\\
			&+b^{2}xuv-\alpha a^{2}xvw-\alpha b^{2}xuv+a^{2}xvw-abxuw+\alpha abxuw\\
			& =-a^{2}xvw+abxuw-b^{2}xuv+\alpha a^{2}xvw-\alpha abxuw+\alpha b^{2}xuv\\
			&+b^{2}xuv-\alpha a^{2}xvw-\alpha b^{2}xuv+a^{2}xvw-abxuw+\alpha abxuw =0
		\end{align*}

\item	$f^{1}(w)=f(w)-\beta_{1}(w)=abxvw-b^{2}xuw+\alpha b^{2}xuw$, for some $\alpha \in \mathbb{Q}$.\\
	$f^{1}(sw)=f(sw)-\beta_{1}(sw)=0$.
	
	We calculate $Df^{1}(sw)$ :
	
 \begin{align*}
			Df^{1}(sw) & =df^{1}(sw)-f^{1}(w)+f^{1}(sdw) =-f^{1}(w)+f^{1}(s(b^{2}-2vx))\\
			&=-f^{1}(w)+bf^{1}(sb)-2vf^{1}(sx)\\
			&  = -abxvw+b^{2}xuw-\alpha b^{2}xuw+b(axvw-bxuw+\alpha bxuw)\\
			& = -abxvw+b^{2}xuw-\alpha b^{2}xuw+abxvw-b^{2}xuw+\alpha b^{2}xuw =0
		\end{align*}
	\end{itemize}
	Thus, $f^{1}$ is a generating class of $Ext_{(\Lambda V,d)}(\mathbb{Q},(\Lambda V,d))$.
\end{example}

\section{Concluding remarks}
\begin{enumerate}
\item
It is worth noting that if  $(\Lambda V,d_{k})$, hence $(\Lambda V,(d_{k})_{\sigma})$, is not elliptic, then  $(f_{k,\sigma}^t)^l(1)$ is not necessarily the representative that gives the exact value of $e_0(X)$. Indeed, as shown in the proof of Theorem \ref{th1.1}, when
the representative $f_{k,\sigma}(1)$ of the top class of $(\Lambda V,(d_{\sigma})_{k})$ is a coboundary, since $(\Lambda V,d_{\sigma})$ is itself elliptic, $(f_{k,\sigma}^t)(1)$
    is  not a coboundary as is $(f_{k,\sigma}^t)^l(1)$  by the convergence of the spectral sequence (\ref{odd spec}). Thus, 
    $[(f_{k,\sigma}^t)^l(1)]$ is the fundamental class of $(\Lambda V,d)$
    
   It follows from the algorithm that, if $(p,q)$ designates the bi-degree  of $f_{k,\sigma}$, then $(f_{k,\sigma}^t)(1)\in \Lambda ^{\geq p}V$ and also $(f_{k,\sigma}^t)^l(1)\in \Lambda ^{\geq p}V$. By \cite[Theorem 2]{B-R}, the fundamental class $\omega \in H^N(\Lambda V,d)$ that results from some class $\omega_{0} \in H^N(\Lambda V,(d_{\sigma})_{k})$  should be   represented by   $\alpha \in \Lambda^{\geq p'}V$ with $p'\geq p$.  Thus $e_0(\Lambda V,d) =p'\geq p$.
       
  \item      
    In the case where $(\Lambda V,(d_{k})_{\sigma})$ is elliptic, $[f_{k,\sigma}(1)]\in H^N(\Lambda V,(d_{\sigma})_{k})$ is the fundamental class of $H^N(\Lambda V,(d_{\sigma})_{k})$. Using the algorithm in \S 4, this persists to yield the fundamental class $[f_{k,\sigma}^t(1)]\in H^N(\Lambda V,d_{\sigma})$, and using the algorithm in \S 3 we get the desired fundamental class $[(f_{k,\sigma}^t)^l]$ of $(\Lambda V,d)$. It follows that $e_0(\Lambda V,(d_{\sigma})_{k}) = e_0(\Lambda V,d_{\sigma}) = e_0(\Lambda V,d)$. This is determined by the explicit formula \cite[Theorem 6]{L-M02}:
    $$e_0(\Lambda V,(d_{\sigma})_{k}) = \dim{V^{odd}} + (k-2)\dim{V^{even}}.$$ 
    This demonstrates the extensibility of our method to the general case.
  \item  
     For details on the requirements imposed to accomplish the algorithm of \S 4, see Example \ref{exp4.5}. From this example, we can see the NP-hard complexity proven in \cite[Theorem 2]{L-M00} when determining this type of invariant.
     \end{enumerate}

\end{document}